\documentclass[preprint]{imsart}

\usepackage[english]{babel}

\usepackage{amsmath}
\usepackage{amsfonts}
\usepackage{amsthm}
\usepackage[numbers]{natbib}
\usepackage{xspace}
\usepackage{graphicx}
\usepackage{subfig}
\usepackage{bm}
\usepackage[colorlinks,citecolor=blue,urlcolor=blue]{hyperref}

\arxiv{1107.6016}

\startlocaldefs

\newenvironment{enumindent}
{\begin{list}
  {\textbf{\arabic{itemcounter}.}}
  {\usecounter{itemcounter}\leftmargin=1.5em\labelsep=0.5em\labelwidth=1em\listparindent=0pt\itemindent=0pt}}
{\end{list}}
\newcounter{itemcounter}

\newtheorem{definition}{Definition}[section]

\newtheorem{lemma}[definition]{Lemma}
\newtheorem{theorem}[definition]{Theorem}
\newtheorem{proposition}[definition]{Proposition}
\newtheorem{corollary}[definition]{Corollary}

\theoremstyle{definition}
\newtheorem{example}{Example}
\newtheorem{remark}[definition]{Remark}

\newcommand{\psd}[1]{{\underline{\bm{#1}}}}

\newcommand{\Fi}{\ensuremath{\mathcal F}\xspace}

\newcommand{\Hi}{\ensuremath{\mathcal H}\xspace}

\newcommand{\Si}{\ensuremath{\mathcal S}\xspace}

\newcommand{\pth}[1]{(#1)}
\newcommand{\pthb}[1]{\bigl(#1\bigr)}
\newcommand{\pthB}[1]{\Bigl(#1\Bigr)}
\newcommand{\pthbb}[1]{\biggl(#1\biggr)}
\newcommand{\pthBB}[1]{\Biggl(#1\Biggr)}
\newcommand{\ptha}[1]{\left(#1\right)}

\newcommand{\bkt}[1]{[#1]}
\newcommand{\bktb}[1]{\bigl[#1\bigr]}
\newcommand{\bktB}[1]{\Bigl[#1\Bigr]}
\newcommand{\bktbb}[1]{\biggl[#1\biggr]}

\newcommand{\bkta}[1]{\left[#1\right]}

\newcommand{\brc}[1]{\{#1\}}
\newcommand{\brcb}[1]{\bigl\{#1\bigr\}}
\newcommand{\brcB}[1]{\Bigl\{#1\Bigr\}}
\newcommand{\brcbb}[1]{\biggl\{#1\biggr\}}

\newcommand{\brca}[1]{\left\{#1\right\}}

\newcommand{\dt}{\ensuremath{\textrm d}\xspace} 
\newcommand{\eqdef}{\overset{\mathrm{def}}{=}}

\newcommand{\scpr}[2]{\left\langle #1,#2 \right\rangle}
\newcommand{\VQ}[1]{\langle#1\rangle}

\newcommand{\ivoo}[1]{\ensuremath{\left(#1\right)}}

\newcommand{\ivfo}[1]{\ensuremath{\left[#1\right)}}
\newcommand{\ivff}[1]{\ensuremath{\left[#1\right]}}

\newcommand{\abs}[1]{\lvert#1\rvert}
\newcommand{\absb}[1]{\bigl\lvert#1\bigr\rvert}

\newcommand{\absbb}[1]{\biggl\lvert#1\biggr\rvert}

\newcommand{\absa}[1]{\left\lvert#1\right\rvert}

\newcommand{\norm}[1]{\lVert#1\rVert}

\newcommand{\floor}[1]{\lfloor#1\rfloor}

\renewcommand{\Pr}{\ensuremath{\mathbb P}\xspace}

\newcommand{\pr}[2][]{\mathbb{P}#1\pth{#2}}
\newcommand{\prb}[2][]{\mathbb{P}#1\pthb{\hspace{1pt}#2\hspace{1pt}}}

\newcommand{\prbb}[2][]{\mathbb{P}#1\pthbb{#2}}

\newcommand{\prcb}[3][]{\mathbb{P}#1\pthb{\hspace{1pt}#2\bigm|#3\hspace{1pt}}}
\newcommand{\prcB}[3][]{\mathbb{P}#1\pthB{#2\Bigm|#3}}

\newcommand{\esp}[2][]{\mathbb{E}#1\bkt{#2}}

\newcommand{\espbb}[2][]{\mathbb{E}#1\bktbb{#2}}

\newcommand{\espcB}[3][]{\mathbb{E}#1\bktB{#2\Bigm|#3}}

\newcommand{\R}{\ensuremath{\mathbf{R}}\xspace}
\newcommand{\Q}{\ensuremath{\mathbf{Q}}\xspace}
\newcommand{\N}{\ensuremath{\mathbf{N}}\xspace}
\newcommand{\Z}{\ensuremath{\mathbf{Z}}\xspace}

\newcommand{\indi}{\ensuremath{\mathbf{1}}\xspace}
\newcommand{\eps}{\varepsilon}
\newcommand{\sbullet}{{\substack{\text{\fontsize{4pt}{4pt}$\bullet$}}}}

\newcommand{\vsp}{\vspace{.15cm}}
\newcommand{\vspar}{\vspace{.25cm}}

\endlocaldefs

\begin{document}

\begin{frontmatter}

\title{2-microlocal analysis of martingales and stochastic integrals}
\runtitle{2-microlocal analysis of martingales and stochastic integrals}

\author{\fnms{Paul} \snm{Balan\c{c}a} 
  \ead[label=e1]{paul.balanca@ecp.fr}
}
\and
\author{\fnms{\ \ Erick} \snm{Herbin} 
  \ead[label=e2]{erick.herbin@ecp.fr} 
  \ead[label=u1,url]{www.mas.ecp.fr/recherche/equipes/modelisation\_probabiliste}
}

\address{\printead{e1,e2}\\ \printead{u1}\\[1em] 
  \'Ecole Centrale Paris and INRIA Regularity team\\ 
  Laboratoire MAS, ECP\\ 
  Grande Voie des Vignes\\ 
  92295 Ch\^atenay-Malabry, France
}

\affiliation{\'Ecole Centrale Paris}
\runauthor{P. Balan\c{c}a and E. Herbin}

\begin{abstract}
  Recently, a new approach in the fine analysis of stochastic processes sample paths has been developed to predict the evolution of the local regularity under (pseudo-)differential operators. In this paper, we study the sample paths of continuous martingales and stochastic integrals. We proved that the almost sure \emph{2-microlocal frontier} of a martingale can be obtained through the local regularity of its quadratic variation. It allows to link the H\"older regularity of a stochastic integral to the regularity of the integrand and integrator processes. These results provide a methodology to predict the local regularity of diffusions from the fine analysis of its coefficients. We illustrate our work with examples of martingales with unusual complex regularity behaviour and square of Bessel processes. 
\end{abstract}



\end{frontmatter}


\section{Introduction}

Sample paths properties of stochastic processes are widely studied since the 1970s (see e.g. \cite{Strassen(1964)}, \cite{Berman(1970)}, \cite{Berman(1972)}, \cite{Orey.Pruitt(1973)} as foundational works). This field of research is still very active and a non-exhaustive list of authors and recent works includes 
Dalang (\cite{Dalang.Nualart.ea(2011)}), Khoshnevisan (\cite{Khoshnevisan.Xiao(2009)},\cite{Dalang.Nualart.ea(2011)}), Lawler (\cite{Lawler(2009)}), Mountford(\cite{Baraka.Mountford.ea(2009)}), Xiao (\cite{Meerschaert.Wu.ea(2008)},\cite{Khoshnevisan.Xiao(2009)},\cite{Baraka.Mountford.ea(2009)}), etc. Among the variety of measures of regularity, pointwise and local H\"older exponents are the most recurrent tools used in the literature. Nevertheless, these two exponents (and as well the moduli of continuity) lack of stability under (pseudo-)differential operators or multiplication by a power function, and therefore do not completely characterize the local regularity of a function or a stochastic process at a given point. 

Simple examples can illustrate these issues. Consider the deterministic "chirp" function
\[
  t\longmapsto f(t) = \abs{t - t_0}^\alpha \sin\pthb{ \abs{t - t_0}^{-\beta} }
\]
where $\alpha$, $\beta$, $t_0$ are positive real numbers. As described in \cite{Echelard(2006)}, this function has a non-trivial regularity at $t_0$. Indeed, its pointwise and local H\"older exponents at $t_0$ (see Section 2 for definitions) are equal to 
\[
  \alpha_{f,t_0} = \alpha\qquad\text{and}\qquad \widetilde\alpha_{f,t_0} = \frac{\alpha}{1+\beta}.
\]
Using notations from \cite{Samko.Kilbas.ea(1993)}, $I_0^\gamma f:t\mapsto\frac{1}{\Gamma(\gamma)}\int_0^t (t-u)^{\gamma-1} f(u)\dt u$ denotes $f$ fractional integration of order $\gamma$. Then, the pointwise and H\"older exponents of $I_0^\gamma f$ at $t_0$ are known to be equal to
\[
  \alpha_{I_0^\gamma f,t_0} = \alpha + \frac{\gamma}{1+\beta}\qquad\text{and}\qquad \widetilde\alpha_{I_0^\gamma f,t_0} = \frac{\alpha}{1+\beta} +\gamma .
\]
Similarly, the study of the local regularity of the function $t\mapsto\abs{t-t_0}^\gamma f(t)$ leads to
\[
  \alpha_{\abs{t - t_0}^\gamma f,t_0} = \alpha + \gamma\qquad\text{and}\qquad \widetilde\alpha_{\abs{t - t_0}^\gamma f,t_0} = \frac{\alpha + \gamma}{1+\beta}.
\]
Hence, in both cases one can observe that the behaviour of local regularity of $I^\gamma_0 f$ and $\abs{t - t_0}^\gamma f(t)$ can not be completely deduced from pointwise and local H\"older exponents of $f$ (it does not correspond to a simple translation of coefficient $\gamma$).

This deterministic example can be easily transposed into a stochastic context using multifractional Brownian motion, a Gaussian process introduced in \cite{Peltier.Levy-Vehel(1995)} and \cite{Benassi.Jaffard.ea(1997)} and applied as probabilistic model in different fields (e.g. \cite{Bianchi.Vieira.ea(2004)} and \cite{Bianchi.Pianese(2008)}). This process, denoted $X^H$ is parametrized by a deterministic function $H:\R\rightarrow\ivoo{0,1}$ and has interesting regularity properties. In the particular case of a so-called regularity function set to $H(t) = a + b\cdot f(t)$ and under some conditions on coefficients $a$ and $b$, the regularity of $X^H$ at $t_0$ almost surely satisfies
\[
  \alpha_{X^H,t_0} = \alpha_{f,t_0}\qquad\text{and}\qquad \widetilde\alpha_{X^H,t_0} = \widetilde\alpha_{f,t_0},
\]
proving that stochastic processes can also have non-trivial behaviours(see \cite{Herbin(2006)}).\vsp

These two different examples illustrate the fact that pointwise and local H\"older exponents are not sufficient to describe entirely the local regularity of a deterministic function or the sample paths of a stochastic process. In this context, \emph{2-microlocal analysis} is a tool that provides a finer characterization. In particular, it allows to describe how pointwise and local exponents evolve under the action of (pseudo-)differential operators and under multiplication by power functions. If it has been first introduced in a deterministic frame (PDE precisely, see \cite{Bony(1986)}), a stochastic approach has been recently developed in \cite{Herbin.Levy-Vehel(2009)}. This previous work exhibited a Kolmogorov-like criterion which gives an almost surely lower bound for the 2-microlocal frontier of a stochastic process. It also focused on the regularity of Gaussian processes at a fixed point $t_0\in\R_+$. Thereby, the regularity of a Gaussian process $X$ at any fixed $t_0$ is almost surely characterized by its incremental variance $\esp{X_t - X_s}^2$. In particular, when considering the following Wiener integral
$
  X_t = \int_0^t \eta(u) \dt W_u,
$
it implies that the regularity of its sample paths is described by the behaviour of the deterministic function $t\mapsto\int_0^t \eta^2(u) \dt u$. 

Given this result, a natural goal is to generalize the statement to any stochastic integral
$
  \displaystyle X_t = \int_0^t H_u \dt M_u,
$
where $H$ is a progressive continuous process and $M$ is a local continuous martingale. In fact, we first prove in Theorem \ref{th:mg_vq} a uniform result on continuous martingales which links up the regularity of the process to the regularity of its quadratic variation. In the specific case of stochastic integrals, this 2-microlocal analysis result can be used to derive local behaviour of sample paths from the regularity of the integrand and the integrator. 

Through these theorems and Examples \ref{ex:mg1}, \ref{ex:mbm_int} and \ref{ex:BESQ}, we show that local regularity of martingales and stochastic integrals can vary along sample paths and may not be deterministic. Similar behaviours have already been exhibited in the literature, from the slow points of Brownian motion (\cite{Orey.Taylor(1974)},\cite{Perkins(1983)}) to more recent work on L\'evy (\cite{Jaffard(1999)}), multifractional (\cite{Ayache.Taqqu(2005)}) and Markov processes (\cite{Xiao(2004)},\cite{Barral.Fournier.ea(2010)}).

Using the 2-microlocal frontier of stochastic integrals, we finally describe how to obtain regularity results for stochastic differential equations. In particular, if it is already known that H\"older regularity of the coefficients have an impact on the existence and uniqueness of solutions (e.g. \cite{Mytnik.Perkins.ea(2006)} in the case of SPDE), we establish for SDE that it also subtly affects the local behaviour of the solution.

The paper is organized as follows: we start by a preliminary section which recalls properties of the classic 2-microlocal frontier in Section \ref{sec:cl_2ml}. Section \ref{sec:ps_2ml} introduces another deterministic tool, the pseudo 2-microlocal frontier, closely related to the previous one. Our main result on the 2-microlocal frontier of continuous martingales is proved in Section \ref{sec:2ml_mg}. Results concerning stochastic integrals and stochastic differential equations are respectively developed in Sections \ref{sec:2ml_int} and \ref{sec:2ml_SDE}. Finally, some technical proofs of deterministic and intermediate results are gathered in \ref{sec:appendix}.


\section{Preliminaries: Classic 2-microlocal analysis} \label{sec:cl_2ml}

The starting point of \emph{2-microlocal analysis} is the definition of specific functional spaces, called \emph{2-microlocal spaces} and denoted by $C^{\sigma,s'}_{t_0}$ where $\sigma,s'\in\R$ and $t_0\in\R$ is a given point. Actually, as noted in \cite{Herbin.Levy-Vehel(2009)}, the study of stochastic processes regularity mainly focuses on spaces $C^{\sigma,s'}_{t_0}$, where $\sigma\in\ivfo{0,1}$ and $s'\in\R$. In this particular case, a continuous function $f$ belongs to $C^{\sigma,s'}_{t_0}$ if there exist $C>0$, $\rho>0$ and a polynomial $P$ such that for all $u,v\in B(t_0,\rho)$,
\begin{equation} \label{eq:def_2ml_spaces}
  \absb{\pthb{ f(u)-P(u) } - \pthb{ (f(v)-P(v) } } \leq C\abs{u-v}^\sigma \pthb{ \abs{u-t_0}+\abs{v-t_0} }^{-s'}.
\end{equation}
$P$ is not necessarily unique, but as proved in \cite{Echelard(2006)}, the Taylor expansion of $f$ of order $\floor{\sigma-s'}$ at $t_0$ can be chosen (if $\floor{\sigma-s'} > 0$, otherwise $P$ is set to $0$). 

Hence, in many situations, the study can even be restricted to spaces $C^{\sigma,s'}_{t_0}$, where $(s',\sigma)\in\sigma_{0,0} = \brc{ (s',\sigma) : \sigma\in\ivfo{0,1}\text{ and } \sigma-s'\in\ivfo{0,1} }$. In this case, as $P=0$, the previous characterization simply becomes
\[
  \forall u,v\in B(t_0,\rho);\quad \abs{f(u)-f(v)} \leq C\abs{u-v}^\sigma \pthb{ \abs{u-t_0}+\abs{v-t_0} }^{-s'}.
\]

Finally, the definition of 2-microlocal spaces when $\sigma\notin\ivff{0,1}$ is slightly more complex and is given at the end of this section since it is little-used in this article.



The pointwise H\"older exponent of a function $f$ is characterized as the supremum of the $\alpha$ such that $f$ belongs to $C^\alpha_{t_0}$. Likewise, the \emph{2-microlocal frontier} of $f$ at $t_0$ is defined as the map $s'\mapsto\sigma_{f,t_0}(s')$ such that
\begin{equation} \label{eq:def_2ml_fr}
  \forall s'\in\R;\quad \sigma_{f,t_0}(s') = \sup\brcB{\sigma\in\R : f\in C^{\sigma,s'}_{t_0}}.
\end{equation}
We note equation \ref{eq:def_2ml_spaces} implies that $C^{\sigma,s'}_{t_0} \subset C^{\sigma',s'}_{t_0}$ when $\sigma'\leq\sigma$, and therefore shows the 2-microlocal frontier exists and is well-defined.

It has been proved (see e.g. \cite{Echelard(2006)} the 2-microlocal frontier $s'\mapsto\sigma_{f,t_0}(s')$ satisfies several interesting properties:
\begin{itemize}
  	\item $\sigma_{f,t_0}(\cdot)$ is a concave and non-decreasing function;
  	\item $\sigma_{f,t_0}(\cdot)$ has left and right derivatives between $0$ and $1$.
  \end{itemize}
Furthermore, if $\alpha_{f,t_0}$ and $\widetilde\alpha_{f,t_0}$ respectively denote pointwise and local  H\"older exponents of $f$ at $t_0$, then definition \eqref{eq:def_2ml_spaces} in the particular cases $\sigma=0$ and $s'=0$ implies
\begin{itemize}
  	\item $\alpha_{f,t_0} = -\inf\brc{s' : \sigma_{f,t_0}(s')\geq 0}$;
  	\item $\widetilde\alpha_{f,t_0} = \sigma_{f,t_0}(0)$
\end{itemize}
with the convention $\alpha_{f,t_0} = +\infty$ if $\sigma_{f,t_0}$ is strictly positive. We note there exist other regularity exponents, like chirp, weak and oscillation exponents introduced in \cite{Arneodo.Bacry.ea(1998)} and \cite{Meyer(1998)} that can be retrieved from the 2-microlocal frontier (see \cite{Echelard(2006)} for an exhaustive list).\vspar

As an example, we consider the Chirp function $f:t\mapsto\abs{t}^\alpha \sin\pthb{\abs{t}^{-\beta}}$ introduced previously.  Its 2-microlocal frontier at $t_0=0$ (see Figure \ref{fig:2ml_chirp}) is
\[
  \forall s'\in\R;\quad \sigma_{f,0}(s') = \frac{s'+\alpha}{1+\beta}.
\]

Stochastic instances can also be exhibited. Indeed, let simply consider the stochastic process $X:t\mapsto B^2_t$ where $B$ is a Brownian motion. Based on results from \cite{Herbin.Levy-Vehel(2009)}, we observe that $X$ almost surely has the following 2-microlocal frontier at $t_0=0$
\[
   \forall s'\in\R;\quad \sigma_{X,0}(s') = \pthb{1+s'}\wedge\frac{1}{2},
\]

\begin{figure}[!ht]
  \subfloat[Chirp function $f:x\mapsto x\sin\pth{x^{-1}}$]{\label{fig:2ml_chirp}\includegraphics[width=0.49\textwidth]{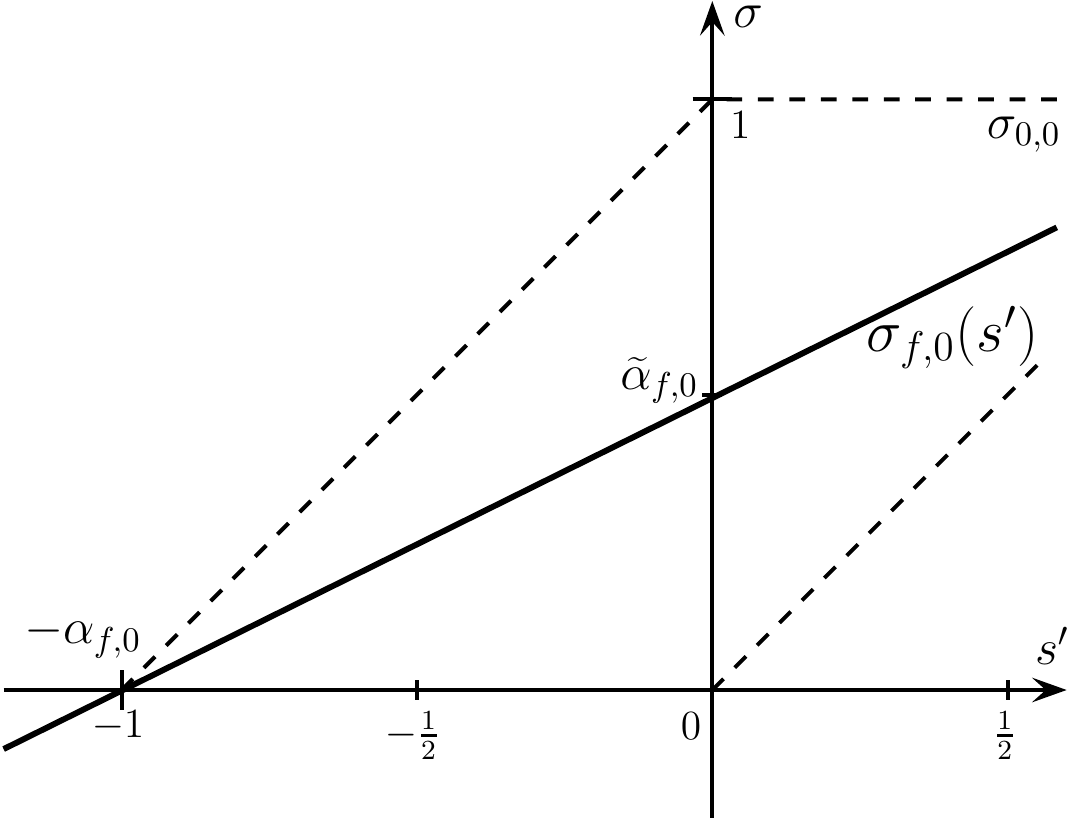}}
  \hfill
  \subfloat[Square of Brownian motion $X:t\mapsto B^2_t$]{\label{fig:2ml_sqbm}\includegraphics[width=0.49\textwidth]{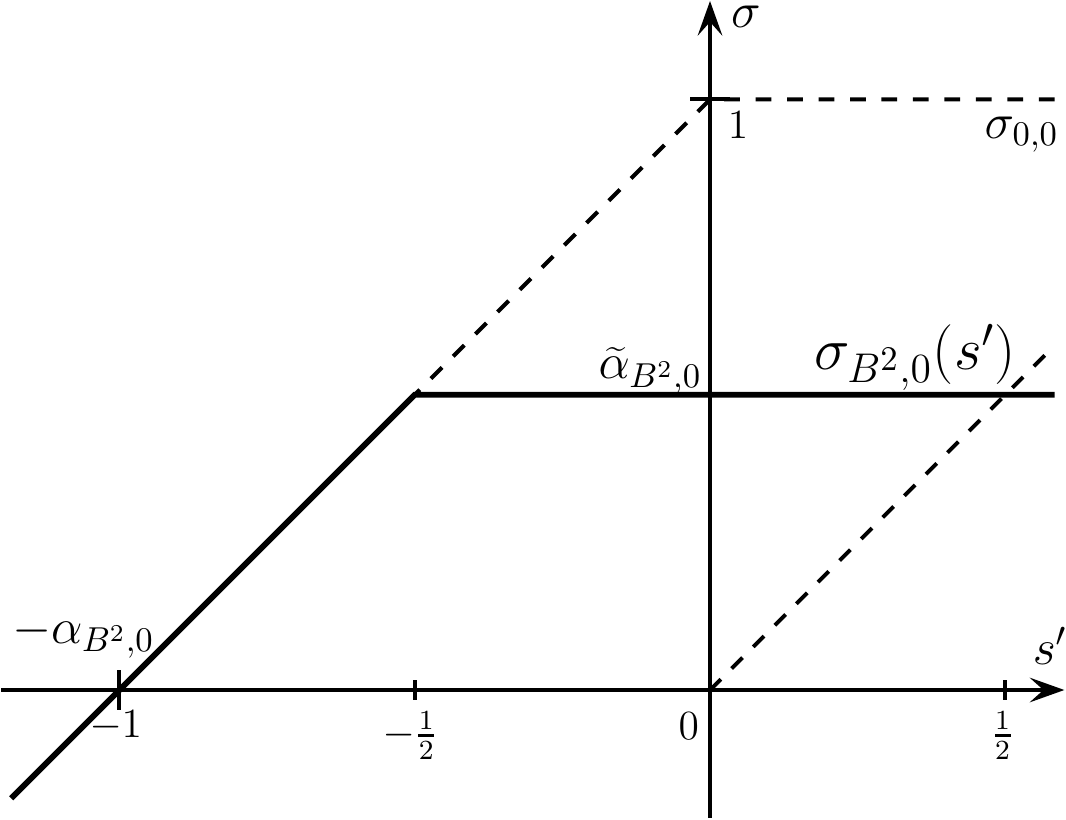}}
  \caption{Examples of 2-microlocal frontiers at $t_0=0$}
  \label{fig:2ml_frontiers}
\end{figure}

These two examples clearly illustrate the fact that pointwise and local H\"older exponents are not sufficient to describe entirely the local regularity of a function, whereas the complete characterization of the 2-microlocal frontier gives a wider insight.

In particular, the knowledge of the 2-microlocal frontier allows to predict the evolution of regularity if a (pseudo-)differential operator is applied to $f$. Let first recall the definition of the fractional integral of order $\alpha\in\R_+$ for a continuous function $f$:
\[
  I^\alpha_{x+}f : t\mapsto \frac{1}{\Gamma(\alpha)} \int_x^t (t-u)^{\alpha-1} f(u)\, \dt u,
\]
where $x\in\R$ is a fixed point. For every $\alpha\in\R_+$, the regularity of $I^\alpha_{x+}f$ satisfies
\[
  \forall s'\in\R;\quad \sigma_{I^\alpha_{x+}f, t_0} (s') = \sigma_{f, t_0}(s') + \alpha,
\]
for all $t_0>x$. Similarly, the function $g_\alpha:t\mapsto \abs{t-t_0}^\gamma (f(t)-f(t_0))$ has its 2-microlocal frontier at $t_0$ equal to
\[
  s'\longmapsto \sigma_{g_\alpha, t_0} (s') = \sigma_{f, t_0}(s'+\alpha).
\]

For sake of completeness, we conclude this section with the definition of 2-microlocal spaces when $\sigma\notin\ivfo{0,1}$. 

For all $\sigma\leq 0$ and $s'\in\R$, a continuous function $f$ is said to belong to $C^{\sigma,s'}_{t_0}$ if there exist $C>0$, $\rho>0$ and a polynomial $P$ such that for all $u,v\in B(t_0,\rho)$
\begin{equation} \label{eq:def_2ml_spaces2}
  \absb{\pthb{ I^{m}_{x+}f(u)-P(u) } - \pthb{ (I^{m}_{x+}f(v)-P(v) } } \leq C\abs{u-v}^{\sigma+m} \pthb{ \abs{u-t_0}+\abs{v-t_0} }^{-s'},
\end{equation}
where $x$ is a fixed point such that $x < t_0$ and $m=-\floor{\sigma}$. 

Similarly, for all $\sigma\geq 1$ and $s'\in\R$, a continuous function $f$ is said to belong to $C^{\sigma,s'}_{t_0}$ if $f$ is differentiable of order $\floor{\sigma}$ around $t_0$ and if there exist $C>0$, $\rho>0$ and a polynomial $P$ such that for all $u,v\in B(t_0,\rho)$
\begin{equation} \label{eq:def_2ml_spaces3}
  \absbb{ \frac{f^{(m)}(u) - P(u)}{\abs{u-t_0}^{\floor{s}-m}} - \frac{f^{(m)}(v) - P(v)}{\abs{v-t_0}^{\floor{s}-m}} } \leq C \abs{u-v}^{\sigma-m} \pthb{ \abs{u-t_0} + \abs{v-t_0} }^{-s'-\floor{s}+m},
\end{equation}
where $s=\sigma-s'$ and $m=\floor{\sigma}$. This last definition is not used through the article since stochastic processes studied are usually almost nowhere differentiable.

The characterization of 2-microlocal spaces given by equations \eqref{eq:def_2ml_spaces}, \eqref{eq:def_2ml_spaces2} and \eqref{eq:def_2ml_spaces3} has been first introduced in \cite{Kolwankar.Vehel(2002)} and then developed in \cite{Seuret.Vehel(2003)} and \cite{Echelard(2006)}. Equivalent definitions based Fourier (\cite{Bony(1986)}) or wavelet (\cite{Jaffard(1991)}, \cite{Meyer(1998)}) transforms have been studied in the literature. Finally, we note that in the case $\sigma\in\N$ or $s\in\N$, our characterization is slightly different from the classical one given in \cite{Bony(1986)}, but this does not affect results developed in this article since the 2-microlocal frontier is not sensitive to this singularity.


\section{Pseudo 2-microlocal analysis} \label{sec:ps_2ml}

According to definitions \eqref{eq:def_2ml_spaces}, \eqref{eq:def_2ml_spaces2} and \eqref{eq:def_2ml_spaces3}, the simple function $t\mapsto (t-t_0)^k$, $k\in\N$, has a 2-microlocal frontier at $t_0$ equal to $+\infty$, on the contrary to the case $k\in\R\setminus\N$.  Intuitively, we understand that 2-microlocal analysis does not take into account the polynomial component of a function around $t_0$. Nevertheless, it might sometimes be necessary and interesting to consider the regularity of polynomials. Indeed, if we simply consider a Brownian motion $B$, we know that its quadratic variation is $\VQ{B}_t = t$, and therefore, if we want to obtain a result which links up regularities of $B$ and $\VQ{B}$, we need a proper tool to characterize variations of $\VQ{B}$.

The concept of \emph{pseudo 2-microlocal analysis} introduced in this section has this purpose. The first step consists in the definition of \emph{pseudo 2-microlocal spaces} $\psd{C}^{s,s'}_{t_0}$ and \emph{pseudo 2-microlocal frontier} $\Sigma_{f,t_0}$ which are similar to classic spaces $C^{s,s'}_{t_0}$ and frontier $\sigma_{f,t_0}$, but which also consider polynomials in the characterization of regularity. Then, properties of this particular frontier are studied, and in particular Theorem \ref{th:classic_pseudo_frontiers} links up classic and pseudo 2-microlocal frontiers.

\subsection{Pseudo 2-microlocal spaces and frontier}

\begin{definition}[Pseudo 2-microlocal spaces] \label{def:pseudo_spaces}
  \hfill
  \begin{enumindent}
    \item Let $\sigma\in\ivfo{0,+\infty}$, $s'\in\R$ and $t_0\in\R$. A continuous function $f$ is said to belong to $\psd{C}^{\sigma,s'}_{t_0}$ if there exist $C>0$ and $\rho>0$ such that for all $u,v\in B(t_0,\rho)$,
    \[
      \abs{f(u) - f(v)} \leq C\abs{u-v}^\sigma \pthb{ \abs{u-t_0}+\abs{v-t_0} }^{-s'}.
    \]
    
    \item Let $\sigma\in\ivoo{-\infty,0}$, $s'\in\R$ and $t_0\in\R$. A continuous function $f$ is said to belong to $\psd{C}^{\sigma,s'}_{t_0}$ if there exist $C>0$ and $\rho>0$ such that for all $u,v\in B(t_0,\rho)$,
    \[
      \absb{ I^{m}_{t_0+}(f-f(t_0))(u) - I^{m}_{t_0+}(f-f(t_0))(v) } \leq C\abs{u-v}^{\sigma+m} \pthb{ \abs{u-t_0}+\abs{v-t_0} }^{-s'},
    \]
    where $m=-\floor{\sigma}$. 
  \end{enumindent}
\end{definition}

We note this definition differs from the characterizations \eqref{eq:def_2ml_spaces} and \eqref{eq:def_2ml_spaces3} of 2-microlocal spaces in the polynomial component which is subtracted in the classic case. Similarly, the pseudo 2-microlocal frontier $\Sigma_{f,t_0}$ is defined by 
\[
  \forall s'\in\R;\quad \Sigma_{f,t_0}(s') = \sup\brcB{\sigma\in\R : f\in \psd{C}^{\sigma,s'}_{t_0}}.
\]
This concept of \emph{pseudo 2-microlocal frontier} has been first introduced in \cite{Herbin.Levy-Vehel(2009)} in the particular case $(s',\sigma)\in\sigma_{0,0}$. The definition above extends it to the whole 2-microlocal domain.

As a corollary, we also define \emph{pseudo pointwise and local H\"older exponents}, which might differ from classic ones.
\begin{definition} \label{def:pseudo_exp}
  Let $f$ be continuous function and $t_0\in\R$. Pseudo pointwise and local H\"older exponents of $f$ at $t_0$ are respectively defined as
  \[
    \psd{\alpha}_{f,t_0} =  \sup\brcbb{\alpha : \limsup_{\rho\rightarrow 0} \sup_{u,v\in B(t_0,\rho)} \frac{\abs{f(u)-f(v)}}{\rho^{\alpha}} < \infty }
  \]
  and
  \[
    \widetilde{\psd{\alpha}}_{f,t_0} =  \sup\brcbb{\alpha : \limsup_{\rho\rightarrow 0} \sup_{u,v\in B(t_0,\rho)} \frac{\abs{f(u)-f(v)}}{\abs{u-v}^{\alpha}} < \infty }.
  \]
  These coefficients satisfy $\psd{\alpha}_{f,t_0} = -\inf\brc{s':\Sigma_{f,t_0}(s')\geq 0}$ and $\widetilde{\psd{\alpha}}_{f,t_0} = \Sigma_{f,t_0}(0)$.
\end{definition}

We observe that the definitions of classic and pseudo 2-microlocal spaces coincide inside the domain $\sigma_{0,0}$ previously introduced. Therefore, if the graph of one of the frontiers belongs to $\sigma_{0,0}$, both frontiers must coincide. For instance, it is the case with sample paths of a Brownian motion $B$ since it has been proved in \cite{Herbin.Levy-Vehel(2009)} that almost surely for all $t\in\R_+$,
\[
  \forall s'\in\R;\quad \sigma_{B,t}(s') = \pthbb{\frac{1}{2}+s'}\wedge\frac{1}{2} \qquad\Rightarrow\qquad \pthb{ s',\sigma_{B,t}(s') }\in\sigma_{0,0}.
\]
Theorem \ref{th:classic_pseudo_frontiers} completely characterizes the link between classic and pseudo 2-microlocal frontiers.

\begin{example} \label{ex:pseudo_power}
  We now illustrate these concepts on a simple example: $f:x\mapsto \abs{x}^\alpha$, where $\alpha > 0$ . We prove that the pseudo 2-microlocal frontier of $f$ at $0$ is equal to
  \[
    \forall s'\in\R;\quad \Sigma_{f,0}(s') = (\alpha+s')\wedge 1. 
  \]
  
  \begin{enumindent}
    \item To obtain the lower bound, we use a simple result from \cite{Echelard(2006)}: there exist $C>0$ and $\rho>0$ such that for all $u,v\in B(0,\rho)$,
    \[
      \abs{ f(u)- f(v) } \leq C \abs{u-v} \pthb{ \abs{u}+\abs{v} }^{\alpha-1},
    \]
    which implies for every $\sigma\in\ivff{0,1}$,
    \begin{align*}
      \abs{ f(u)- f(v) } 
      &= \abs{ f(u)- f(v) }^\sigma \cdot \abs{ f(u)- f(v) }^{1-\sigma} \\
      &\leq C \abs{u-v}^\sigma \pthb{ \abs{u}+\abs{v} }^{\sigma(\alpha-1)} \cdot \pthb{ \abs{u}+\abs{v} }^{\alpha(1-\sigma)} \\
      &= C \abs{u-v}^\sigma \pthb{ \abs{u}+\abs{v} }^{\alpha-\sigma}.
    \end{align*}
    Thus, for all $s'\in\ivff{-\alpha,1-\alpha}$, $f$ belongs to $\psd{C}^{\sigma,s'}_{0}$ when $\sigma \leq (s'+\alpha)\wedge 1$.
    
    \item On the other side, we note that for all $u\in\R$, $\abs{f(u)-f(0)} = \abs{u}^\alpha = \abs{u}^{\alpha+s'} \abs{u}^{-s'}$, and therefore $\Sigma_{\abs{x}^\alpha,0}(s') \leq \alpha+s'$.
    Furthermore, as $f$ is differentiable, for all $s'\in\R$, $\rho>0$ and $\eps>0$, we know that
    \[
      \sup_{u,v\in B(0,\rho)} \frac{ \abs{f(u) - f(v)} }{ \abs{u-v}^{1+\eps} \pthb{\abs{u}+\abs{v}}^{-s'} } = +\infty,
    \]
    and therefore $\Sigma_{f,0}(s') \leq 1$.
  \end{enumindent}
\end{example}
If we compare classic and pseudo 2-microlocal frontiers of $x\mapsto\abs{x}^\alpha$, we notice that for all $s'\in\R$
\[
  \Sigma_{\abs{x}^\alpha,0}(s') = (\alpha+s')\wedge 1 \quad\text{and}\quad \sigma_{\abs{x}^\alpha,0}(s') = 
  \begin{cases}
    \,\alpha + s' & \text{if } \alpha\in\R_+\setminus\N; \\
    \,+\infty & \text{if } \alpha\in\N.
  \end{cases}
\]
Therefore, we have shown the pseudo 2-microlocal frontier indeed takes into account polynomials and does not make any distinction between integer and non-integers powers in terms of regularity. Furthermore, this example also illustrates the fact that classic and pseudo 2-microlocal frontiers do not coincide in general.

\begin{remark}
  We recalled previously that there exist characterizations of 2-microlocal spaces using Wavelet (or Fourier) transform. Hence, we know from \cite{Jaffard(1991)} that $f$ belongs to $C^{\sigma,s'}_{t_0}$ if and only if
  \[
    \forall j,k\in\Z \text{ s.t. } \abs{t_0-k2^{-j}}\leq 1; \quad \abs{d_{j,k}} \leq C 2^{-j\sigma}\pthb{2^{-j}+\abs{k2^{-j}-t_0}}^{-s'},
  \]
  where $N > \max(\sigma,\sigma-s')$, $d_{j,k} = 2^j\scpr{f}{\psi(2^jx - k)}$, $\psi\in\Si(\R)$ has $N$ vanishing moments and is such that $\brcb{\psi_{j,k} = 2^{j/2}\psi\pth{2^jx-k} }_{(j,k)\in\Z^2}$ forms an orthonormal basis of $L^2(\R)$.
  
  Therefore, a natural question is to wonder if this characterization can be adapt to pseudo 2-microlocal spaces. In fact, a simple calculation proves that if we replace the vanishing moments hypothesis by
  \[
    \forall k\in\brcb{ 1,\dotsc,N}; \quad \scpr{x^k}{\psi} \neq 0,
  \]
  where as previously $N > \max(\sigma,\sigma-s')$, then this definition becomes a characterization of pseudo 2-microlocal spaces.
\end{remark}

\subsection{Properties}

In this second part, we prove a few important results related to the pseudo 2-microlocal frontier and which are useful later in the article.
The main following theorem gives a general formula which links up classic and pseudo 2-microlocal frontiers for continuous functions.
\begin{theorem} \label{th:classic_pseudo_frontiers}
  Let $f$ be a continuous function and $t_0$ be in \R. We define $p_{f,t_0}$ as the integer
  \[
    p_{f,t_0} = \inf\brcb{n\geq 1 : f^{(n)}(t_0)\text{ exists and }f^{(n)}(t_0)\neq 0},
  \]
  with the usual convention $\inf\brc{\emptyset} = +\infty$.
  
  Then, the pseudo 2-microlocal frontier of $f$ at $t_0$ is equal to
  \[
    \forall s'\in\R ;\quad \Sigma_{f,t_0}(s') = \sigma_{f,t_0}(s')\wedge\pth{s'+p_{f,t_0}}\wedge 1,
  \]
  unless $f$ is locally constant at $t_0$, which implies in that specific case:
  \[
    \forall s'\in\R ;\quad \Sigma_{f,t_0}(s') = \sigma_{f,t_0}(s') = +\infty.
  \]
\end{theorem}

We note that this theorem is consistent with our previous calculation of the pseudo frontier of $x\mapsto\abs{x}^\alpha$ in Example \ref{ex:pseudo_power}. Furthermore, properties on the map $s'\mapsto\Sigma_{f,t_0}(s')$ can be deduced as a simple corollary.
\begin{corollary} \label{cor:pseudo_properties}
  Similarly to the classic frontier, the pseudo 2-microlocal frontier of $f$ at $t_0$ satisfies
  \begin{itemize}
  	\item $\Sigma_{f,t_0}$ is a concave and non-decreasing function;
  	\item $\Sigma_{f,t_0}$ has left and right derivatives between $0$ and $1$.
  \end{itemize}
\end{corollary}

Using the previous theorem, we can also obtain a useful result which illustrates the behaviour of the pseudo 2-microlocal frontier when a function is integrated.
\begin{theorem} \label{th:int_pseudo_frontier}
  Let $f$ be a continuous function and $F$ be
  \[
    \forall t\in\R;\quad F(t) = \int_0^t f(s) \,\dt s.
  \]
  Then, for any $t_0$ in \R, the pseudo 2-microlocal frontier of $F$ at $t_0$ is equal to:
  \[
    \forall s'\in\R;\quad \Sigma_{F,t_0}(s') = 
    \begin{cases}
      \pthb{1+s'} \wedge 1 & \text{if } f(t_0)\neq 0; \\
      \pthb{ \Sigma_{f,t_0}(s')+1 } \wedge 1 & \text{if } f(t_0)=0 \text{ and is not locally constant}; \\
      +\infty & \text{if $f$ is locally equal to 0 at } t_0,
    \end{cases}
  \]
\end{theorem}
For sake of readability, technical proofs of Theorems \ref{th:classic_pseudo_frontiers} and \ref{th:int_pseudo_frontier} are given in Appendix \ref{sec:appendix}.


To end this section, we establish a result on the pseudo 2-microlocal frontier of composed functions, which is necessary in Section \ref{sec:2ml_mg} to determine the regularity of martingales.
\begin{proposition} \label{prop:pseudo_2ml_time_change}
  Let $f$ and $g$ be two continuous functions and let $h$ be the composition $g \circ f$. Then, for every $t\in\R$, the pseudo 2-microlocal frontier $\Sigma_{h,t}$ of $h$ at $t$ satisfies the inequality
  \[
    \Sigma_{h,t}\pthb{ s'_f \cdot \Sigma_{g,f(t)}(s'_g) + s'_{g} \cdot \psd{\alpha}_{f,t} } \geq \Sigma_{g,f(t)}(s'_g) \cdot \Sigma_{f,t}(s'_f),
  \]
  for all $s'_f \in \ivfo{-\psd{\alpha}_{f,t},+\infty}$ and $s'_g \in [-\psd{\alpha}_{g,f(t)},0]$.
\end{proposition}
\begin{proof}
  Let $t\in\R$, $s'_f \geq -\psd{\alpha}_{f,t}$ and $s'_g \in [-\psd{\alpha}_{g,f(t)},0]$.
  For every $\eps > 0$, there exist $C_f,C_g>0$ and $\rho>0$ such for all $u,v\in B(t,\rho)$,
  \[
    \abs{ f(u) - f(v) } \leq C_f \abs{ u-v }^{ \Sigma_{f,t}(s'_f) - \eps } \pthb{ \abs{u-t}+\abs{v-t} }^{ -s'_f }
  \]
  and for all $x,y\in B(f(t),\rho)$
  \[
    \abs{ g(x)-g(y) } \leq C_g \abs{ x-y }^{ \Sigma_{g,f(t)}(s'_g) - \eps } \pthb{ \abs{x-f(t)}+\abs{y-f(t)} }^{ -s'_g }.
  \]  
  Therefore, we obtain
  \begin{align*}
    &\abs{ h(u) - h(v) } \\
    &= \abs{ (g \circ f)(u) - (g \circ f)(v) } \\
    &\leq C_g \abs{ f(u) - f(v) }^{\Sigma_{g,f(t)}(s'_g) - \eps} \pthb{ \abs{f(u) - f(t)} + \abs{f(v) - f(t)} }^{-s'_g} \\
    &\leq C \abs{ u - v }^{(\Sigma_{g,f(t)}(s'_g) - \eps)(\Sigma_{f,t}(s'_f) - \eps)} \pthb{ \abs{u - t} + \abs{v - t} }^{-s'_f\cdot (\Sigma_{g,f(t)}(s'_g) - \eps) - s'_g\cdot (\psd{\alpha}_{f,t}-\eps)},
  \end{align*}
  since $\abs{f(u)-f(t)} \leq C_f \abs{u-t}^{\psd{\alpha}_{f,t}-\eps}$ and $-s'_g \geq 0$.
  Then, using the continuity of the pseudo 2-microlocal frontier, we get the expected inequality.
\end{proof}

Using the previous proposition, we can obtain simpler inequalities, which nevertheless might be less accurate with specific functions $f$ and $g$.
\begin{corollary} \label{cor:pseudo_2ml_time_change}
  Let $f$ and $g$ be two continuous functions and let $h$ be the composition $g \circ f$. Then, for any $x\in\R$, the pseudo 2-microlocal frontier $\Sigma_{h,x}$ of $h$ at $x$ satisfies the inequalities:
  \[
    \forall s'\geq-\psd{\alpha}_{h,x};\quad \Sigma_{h,x}(s') \geq \widetilde{\psd{\alpha}}_{g,f(x)} \cdot \Sigma_{f,x}(s' / \widetilde{\psd{\alpha}}_{g,f(x)} ),
  \]
  and
  \[
    \forall s'\geq-\psd{\alpha}_{h,x};\quad \Sigma_{h,x}(s') \geq \widetilde{\psd{\alpha}}_{f,x} \cdot \Sigma_{g,f(x)}(s' / \psd{\alpha}_{f,x} ).
  \]
  In particular, if we consider the pseudo pointwise and local H\"older exponents, we obtain:
  and thus, in particular of pseudo pointwise exponents:
  \[
    \psd{\alpha}_{h,x} \geq \psd{\alpha}_{g,f(x)} \cdot \psd{\alpha}_{f,x}\qquad\text{and}\qquad \widetilde{\psd{\alpha}}_{h,x} \geq \widetilde{\psd{\alpha}}_{g,f(x)} \cdot \widetilde{\psd{\alpha}}_{f,x}
  \]
\end{corollary}
\begin{proof}
  The first two inequalities are obtained using the previous Proposition \ref{prop:pseudo_2ml_time_change} in the particular cases $s'_g = 0$ and $s'_f = 0$. 
  The inequalities on pseudo exponents are deduced from the previous ones, when $s' = \psd{\alpha}_{g,f(x)} \cdot \psd{\alpha}_{f,x}$ and $s'=0$.
\end{proof}


\section{Stochastic 2-microlocal analysis of martingales} \label{sec:2ml_mg}

In the remaining of the article, we study the 2-microlocal frontier of stochastic processes. As noted in \cite{Herbin.Levy-Vehel(2009)}, when we consider a random process $X$, classic $\sigma_{X(\omega),t_0}(s',\omega)$ and pseudo $\Sigma_{X(\omega),t_0}(s',\omega)$ 2-microlocal frontiers at $t_0$ become random functions. The stochastic 2-microlocal frontier $(s',\omega)\mapsto\sigma_{X(\omega),t_0}(s',\omega)$ is clearly measurable as it is continuous on the variable $s'$ and the process $X$ is progressive.

As previously announced, we first study in Theorem \ref{th:mg_vq} the pseudo 2-microlocal frontier of continuous martingales. We deduce from this result uniform lower bounds for semimartingales in Proposition \ref{prop:semiMgs}. Then, we construct an example of martingale to exhibit interesting regularity properties. Finally, we show how Theorem \ref{th:mg_vq} can be easily extended to time changed multifractional Brownian motion.

\subsection{2-microlocal frontier of continuous martingales and semimartingales}

In this part, we study the regularity of continuous local martingales in order to obtain a link between pseudo 2-microlocal frontiers of the martingale on one hand and it quadratic variation on the other hand. The following theorem states our main result.
\begin{theorem} \label{th:mg_vq}
  Let $M$ be a continuous local martingale and $\VQ{M}$ be its quadratic variation. Then, almost surely for all $t\in\R_+$, the pseudo 2-microlocal frontiers of $M$ at $t$ satisfies the following equality:
  \begin{equation} \label{eq:mg_vq}
    \forall s'\geq-\psd{\alpha}_{M,t};\quad  \Sigma_{M,t}(s') = \frac{1}{2}\Sigma_{\VQ{M},t}\ptha{{2s'}}.
  \end{equation}
  Consequently, pseudo pointwise and local H\"older exponents are equal to
  \[
    \psd{\alpha}_{M,t} = \frac{\psd{\alpha}_{\VQ{M},t}}{2} \quad\text{and}\quad \widetilde{\psd{\alpha}}_{M,t} = \frac{\widetilde{\psd{\alpha}}_{\VQ{M},t}}{2}.
  \]
\end{theorem}
For sake of readability, we divide the proof of expression \eqref{eq:mg_vq} in two parts (lower and upper bounds).

\subsubsection{Proof of the lower bound}

  Let first prove that almost surely for all $t\in\R_+$, the pseudo 2-microlocal frontier of $M$ at $t$ satisfies
  \[
    \forall s'\geq-\psd{\alpha}_{M,t};\quad  \Sigma_{M,t}(s') \geq \frac{1}{2}\Sigma_{\VQ{M},t}\ptha{{2s'}}.
  \]
  Without any restriction, we can suppose that $M_0=0$. Let set for all $t\in\R_+$
  \[
    T_t = \inf\brc{s : \VQ{M}_s > t}.
  \]
  Then, according to the extended Dubins-Schwarz Theorem ($5.1.7$ in \cite{Revuz.Yor(1999)}), we know there exists an enlargement $(\widetilde\Omega,\widetilde\Fi_t,\widetilde\Pr)$ of the probability space $(\Omega,\Fi_{T_t},\Pr)$ and a Brownian motion $\widetilde\beta$ on $\widetilde\Omega$ such that the process
  \[
    B_t = M_{T_t} + \int_0^t \indi_{\brc{s>\VQ{M}_\infty}}\dt \widetilde\beta_s
  \]
  is a $(\widetilde\Fi_t)_t$ Brownian motion and for every $t\in\R_+$, $M_t = B_{\VQ{M}_t}$. We note that the enlargement has the following form,
  \[
    \widetilde\Omega = \Omega\times\Omega',\ \widetilde\Fi_t = \Fi_{T_t}\otimes\Fi_t' \text{ and } \widetilde\Pr = \Pr\otimes\Pr'.
  \]
  
  Based on results from \cite{Herbin.Levy-Vehel(2009)}, it is known that $\widetilde\Pr$-almost surely for all $t\in\R_+$, classic and pseudo local H\"older exponents of the Brownian motion $B$ are equal to:
  \[
    \widetilde{\psd{\alpha}}_{B,t} = \widetilde\alpha_{B,t} = \frac{1}{2}.
  \]
  Then, as $M$ corresponds to the composition of $B$ and $\VQ{M}$, from Corollary \ref{cor:pseudo_2ml_time_change}, we obtain $\widetilde\Pr$-almost surely for all $t\in\R_+$,
  \[
    \forall s'\geq-\psd{\alpha}_{M,t};\quad  \Sigma_{M,t}(s') \geq \widetilde{\psd{\alpha}}_{B,t} \cdot\Sigma_{\VQ{M},t}\ptha{{s'/\widetilde{\psd{\alpha}}_{B,t}}} =  \frac{1}{2}\Sigma_{\VQ{M},t}\ptha{{2s'}}.
  \]

  More precisely, there exists $\widetilde\Omega_0\in\widetilde\Fi$ such that $\widetilde\Pr(\widetilde\Omega_0)=1$ and for all $\widetilde\omega\in\widetilde\Omega_0$ and $t\in\R_+$, the pseudo 2-microlocal frontier of $M_\cdot(\widetilde\omega)$ at $t$ satisfies
  \[
    \forall s'\geq-\psd{\alpha}_{M(\widetilde\omega),t}(\widetilde\omega);\quad  \Sigma_{M(\widetilde\omega),t}(s',\widetilde\omega) \geq \frac{1}{2}\Sigma_{\VQ{M}(\widetilde\omega),t}\ptha{{2s',\widetilde\omega}}.
  \]
  Let $\omega\in\Omega$, according to the definition of the enlargement $\widetilde\Omega$, we know that for all $\omega'\in\Omega'$, 
  \[
    \Sigma_{M,t}(s',(\omega,\omega'))=\Sigma_{M,t}(s',\omega) \quad\text{ and }\quad \Sigma_{\VQ{M},t}(s',(\omega,\omega'))=\Sigma_{\VQ{M},t}(s',\omega).
  \]
  Therefore, the set $\widetilde\Omega_0$ has the form $\widetilde\Omega_0=\Omega_0\times\Omega'$ where $\Omega_0\in\Fi$ and $\pr{\Omega_0}=1$. 
  
  Hence, we have proved that $\Pr$-almost surely for all $t\in\R_+$,
  \[
    \forall s'\geq-\psd{\alpha}_{M,t};\quad  \Sigma_{M,t}(s') \geq \frac{1}{2}\Sigma_{\VQ{M},t}\ptha{{2s'}}.
  \]

\subsubsection{Proof of the upper bound}

  The second step is to establish that almost surely for all $t\in\R_+$, the pseudo 2-microlocal frontier of $M$ at $t$ satisfies
  \[
    \forall s'\geq-\psd{\alpha}_{M,t};\quad  \Sigma_{M,t}(s') \leq \frac{1}{2}\Sigma_{\VQ{M},t}\ptha{{2s'}}.
  \]
  As previously, using Dubins-Schwarz Theorem, there exists a Brownian motion $B$ such that almost surely for all $t\in\R_+$, we have $M_t = B_{\VQ{M}_t}$.
  Then, to prove the upper bound, we use the following technical lemma satisfied by the Brownian motion $B$.
  \begin{lemma} \label{lemma:bm_incr}
    Let $B$ be a Brownian motion. Then, there exists an event $\Omega_0$ such that $\pr{\Omega_0}=1$ and for all $\omega\in\Omega_0$, $N\in\N$, $\eps>0$, there exists $h(\omega)>0$ such that for all $\rho\leq h(\omega)$ and $t\in\ivff{0,N}$, we have
    \[
      \sup_{u,v\in B(t,\rho)} \brcb{ \abs{B_u - B_v} } \geq \rho^{1/2 + \eps}.
    \]
  \end{lemma}
  This lemma is a corollary of a more general result proved on the multifractional Brownian motion at the end of this section (see Proposition \ref{prop:mbm_incr}).
  
  Therefore, let set $\omega\in\Omega_0$, $N\in\N$ and $\eps>0$. We denote by $\VQ{M}^N$ the increasing process  $\VQ{M}^N_t = \VQ{M}_t\wedge N$ and $M^N$ the compound process $M^N_t = B_{\VQ{M}^N_t}$. According to the definition of the pseudo 2-microlocal frontier, for $s'\in\R$ and $t\in\R_+$, there exist sequences $(s_n(\omega))_n,(t_n(\omega))_n$ such that
  \[
     \lim_{n\rightarrow+\infty}  s_n(\omega) = \lim_{n\rightarrow+\infty} t_n(\omega) = t
  \]
  and for all $n\in\N$,
  \begin{equation} \label{eq:pseudo_vq}
     \frac{\absb{ \VQ{M}^N_{t_n} - \VQ{M}^N_{s_n} }}{\abs{t_n-s_n}^{\Sigma_{\VQ{M}^N,t}(s')+\eps} \pthb{ \abs{t - t_n}+\abs{t - s_n} }^{-s'}}(\omega) \geq 1.
  \end{equation}
  Without loss of generality, we can also suppose for all $n\in\N$ that $s_n(\omega)\leq t_n(\omega)$ and $\abs{ \VQ{M}^N_{t_n}(\omega) - \VQ{M}^N_{s_n}(\omega) } \leq h(\omega)$, where $h(\omega)$ is defined in Lemma \ref{lemma:bm_incr}.
  
  Then, since $\VQ{M}^N_\cdot(\omega)\leq N$, for each $n\in\N$ there exists $u_n(\omega),v_n(\omega)\in\ivff{\VQ{M}^N_{s_n},\VQ{M}^N_{t_n}}$ such that
  \begin{equation} \label{eq:inc_bm}
    \abs{ B_{u_n} - B_{v_n} } \geq \absb{ \VQ{M}^N_{t_n} - \VQ{M}^N_{s_n} }^{1/2+\eps}.
  \end{equation}
  As $t\mapsto\VQ{M}_t(\omega)$ is a continuous non-decreasing function, there exist $x_n(\omega),y_n(\omega)$ such that $\VQ{M}^N_{x_n} = u_n$ and $\VQ{M}^N_{y_n} = v_n$ and $\ivff{ x_n,y_n }\subset\ivff{ s_n,t_n }$.
  
  \noindent Then, using inequalities \eqref{eq:pseudo_vq} and \eqref{eq:inc_bm}, we obtain
  \begin{align*}
    \absb{ M^N_{x_n} - M^N_{y_n}  } 
    &= \absb{ B_{u_n} - B_{v_n} } \\
    &\geq \absb{ \VQ{M}^N_{t_n} - \VQ{M}^N_{s_n} }^{1/2+\eps} \\
    &\geq \abs{t_n-s_n}^{(\Sigma_{\VQ{M}^N,t}(s')+\eps)\cdot(1/2+\eps)} \pthb{ \abs{t - t_n}+\abs{t - s_n} }^{-s'\cdot(1/2+\eps)}.
  \end{align*}
  Let now distinguish the two different cases. 
  
  \begin{enumindent}
    \item If $s'\leq 0$, we note that $\abs{x_n - y_n} \leq \abs{s_n - t_n}$ and $\abs{t - x_n}+\abs{t - y_n} \leq 2\pthb{ \abs{t - t_n}+\abs{t - s_n} }$. Hence, there exists $C>0$ such that for all $n\in\N$,
    \begin{align*}
      \abs{ M^N_{x_n} - M^N_{y_n}  } 
      \geq C\abs{x_n-y_n}^{(\Sigma_{\VQ{M}^N,t}(s')+\eps)\cdot(1/2+\eps)} \pthb{ \abs{t - x_n}+\abs{t - y_n} }^{-s'\cdot(1/2+\eps)},
    \end{align*}
    since $-s'\cdot(1/2+\eps) \geq 0$ and $(\Sigma_{\VQ{M}^N,t}(s')+\eps)\cdot(1/2+\eps) \geq 0$.
  
  \item If $s'\geq 0$, up to an extraction, we can suppose the following convergence
  \[
    \lim_{n\rightarrow\infty}  \frac{ \abs{t_n-s_n} }{ \abs{t - t_n}+\abs{t - s_n} } = m,
  \]
  where $m\in\ivff{0,1}$ since the sequence is positive and bounded by $1$.
   
  \begin{itemize}
	  \item If $m = 0$, as $\ivff{ x_n,y_n }\subset\ivff{ s_n,t_n }$, we obtain
    \[
	    \pthb{ \abs{t - t_n}+\abs{t - s_n} } \sim_{n\rightarrow\infty} \pthb{ \abs{t - x_n}+\abs{t - y_n} }.
	  \] 
	  Therefore, since we also have $\abs{x_n-y_n}\leq\abs{s_n-t_n}$, for all $n\in\N$ large enough, we get
    \begin{align*}
      \absb{ M^N_{x_n} - M^N_{y_n}  }
      \geq C \abs{x_n-y_n}^{(\Sigma_{\VQ{M}^N,t}(s')+\eps)\cdot(1/2+\eps)} \pthb{ \abs{t - x_n}+\abs{t - y_n} }^{-s'\cdot(1/2+\eps)},
    \end{align*}
    where $C$ is a positive constant.
   
    \item If $m > 0$, we know that $\abs{t_n-s_n} \sim_n m\cdot\pthb{ \abs{t - t_n}+\abs{t - s_n }}$. Therefore, from \eqref{eq:pseudo_vq}, there exists $C>0$ such that for all $n\in\N$ large enough
    \[
      \absb{ \VQ{M}^N_{t_n} - \VQ{M}^N_{s_n} } \geq C \pthb{ \abs{t - t_n}+\abs{t - s_n} }^{\Sigma_{\VQ{M}^N,t}(s')+\eps-s'},
    \]
    which necessarily implies that $\Sigma_{\VQ{M}^N,t}(s') = \psd{\alpha}_{\VQ{M}^N,t} + s'$, according to Definition \ref{def:pseudo_exp} for $\psd{\alpha}_{\VQ{M}^N,t}$.
    Then, for all $n\in\N$ large enough, we obtain
    \begin{align*}
      \absb{ M^N_{x_n} - M^N_{y_n}  } 
      &\geq \abs{t_n-s_n}^{(\Sigma_{\VQ{M}^N,t}(s')+\eps)\cdot(1/2+\eps)} \pthb{ \abs{t - t_n}+\abs{t - s_n} }^{-s'\cdot(1/2+\eps)} \\
      &\geq C\abs{t_n-s_n}^{(\Sigma_{\VQ{M}^N,t}(s')+\eps-s')\cdot(1/2+\eps)} \\
      &= C\abs{t_n-s_n}^{(\psd{\alpha}_{\VQ{M}^N,t}+\eps)\cdot(1/2+\eps)} \\
      &\geq C\abs{x_n-y_n}^{(\psd{\alpha}_{\VQ{M}^N,t}+\eps)\cdot(1/2+\eps)} \qquad\text{as $\psd{\alpha}_{\VQ{M}^N,t}+\eps > 0$} \\
      &= C\abs{x_n-y_n}^{(\psd{\alpha}_{\VQ{M}^N,t}+\eps-s')\cdot(1/2+\eps)} \cdot \abs{x_n-y_n}^{-s'\cdot(1/2+\eps)} \\
      &\geq C\abs{x_n-y_n}^{(\Sigma_{\VQ{M}^N,t}(s')+\eps)\cdot(1/2+\eps)} \pthb{ \abs{t - x_n}+\abs{t - y_n} }^{-s'\cdot(1/2+\eps)},
    \end{align*}
    since $-s'\leq 0$ and $\abs{x_n-y_n}\leq \pthb{ \abs{t - x_n}+\abs{t - y_n} }$.
    \end{itemize}
  \end{enumindent}
  To summarize, in each case, we have proved that for all $s'\geq-\psd{\alpha}_{M,t}$ there exists $C>0$ such that for all $n\in\N$ large enough,
  \begin{align*}
    \abs{ M^N_{x_n} - M^N_{y_n}  } 
    \geq C\abs{x_n-y_n}^{(\Sigma_{\VQ{M}^N,t}(s')+\eps)\cdot(1/2+\eps)} \pthb{ \abs{t - x_n}+\abs{t - y_n} }^{-s'\cdot(1/2+\eps)},
  \end{align*}
  which proves that for all $N\in\N$, $\eps>0$ and $s'\in\ivff{-\psd{\alpha}_{M,t},0}$, with probability one
  \[
    \forall t\in\R_+;\quad \Sigma_{M^N,t}\ptha{s'\cdot(1/2+\eps)} \leq (\Sigma_{\VQ{M}^N,t}(s')+\eps)\cdot(1/2+\eps).
  \]
  We observe that for all $\omega\in\Omega_0$ and $t\in\R_+$, there exists $N_0(\omega)$ such that for all $N \geq N_0(\omega)$,
  \[
    \forall s'\in\R;\quad \Sigma_{M^N,t}(s') = \Sigma_{M,t}(s')\qquad\text{and}\qquad \Sigma_{\VQ{M}^N,t}(s') = \Sigma_{\VQ{M},t}(s').
  \]
  Therefore, as $N\rightarrow+\infty$, we obtain with probability one
  \[
    \forall t\in\R_+;\quad \Sigma_{M,t}\ptha{s'\cdot(1/2+\eps)} \leq (\Sigma_{\VQ{M},t}(s')+\eps)\cdot(1/2+\eps).
  \]
  Then, using a sequence $(\eps_n)_{n\in\N}$ which converges to zero and $s'\in\Q$ (sufficient as the pseudo 2-microlocal frontier is continuous), we get the expected inequality. With probability one and for all $t\in\R_+$,
  \[
    \forall s'\geq-\psd{\alpha}_{M,t};\quad  \Sigma_{M,t}(s') \leq \frac{1}{2}\Sigma_{\VQ{M},t}\ptha{{2s'}}.
  \]
\begin{flushright}\qedsymbol\end{flushright}

\begin{remark}
  In this context of martingales, one could imagine to introduce a more general definition of the 2-microlocal frontier where the Euclidean metric is replaced by any distance on $\R_+$. In particular, we could choose the random metric $d_{\VQ{M}}(x,y) = \abs{\VQ{M}_x - \VQ{M}_y}$ (which is in fact a pseudo-metric) and characterize the pseudo 2-microlocal local frontier $\Sigma_{M,t}^{d_{\VQ{M}}}$ of the martingale $M$ with respect to this metric. Using Dubins-Schwarz, we can easily show that in this case, we would obtain:
  \[
    \forall s'\in\ivff{-\textstyle\frac{1}{2},0};\quad \Sigma_{M,t}^{d_{\VQ{M}}}(s') = \frac{1}{2} + s'.
  \]
  However, on the contrary to Theorem \ref{th:mg_vq}, this expression of the 2-microlocal frontier of $M$ does not capture all the regularity of $M$ with respect the Euclidean metric.
  
\end{remark}\vspace*{1em}

Theorem \ref{th:mg_vq} characterizes the regularity of a martingale $M$ in terms of pseudo 2-microlocal frontier. Let now gives a criterion which allows to extend the equality to the classic frontier of $M$.
\begin{corollary} \label{cor:eq_cl_psd}
  Let $M$ be a continuous local martingale and $\VQ{M}$ be its quadratic variation. Then, with probability one, for all $t\in\R$ and $s'\geq-\psd{\alpha}_{M,t}$ such that 
  \[
    \Sigma_{M,t}(s') < \pth{2+s'},
  \]
  the pseudo 2-microlocal frontier of $M$ at $t$ satisfies
  \[
    \sigma_{M,t}(s') = \Sigma_{M,t}(s') = \frac{1}{2}\Sigma_{\VQ{M},t}\ptha{{2s'}}.
  \]
  In particular, the local H\"older exponents of $M$ and $\VQ{M}$ satisfy almost surely
  \[
    \forall t\in\R_+;\quad \widetilde{\alpha}_{M,t} = \widetilde{\psd{\alpha}}_{M,t} = \frac{\widetilde{\psd{\alpha}}_{\VQ{M},t}}{2}.
  \]
\end{corollary}
\begin{proof}
  According to Theorem \ref{th:classic_pseudo_frontiers}, we know that almost surely for all $t\in\R_+$ when $M$ is not locally constant at $t$,
  \[
    \forall s'\in\R ;\quad \Sigma_{M,t}(s') = \sigma_{M,t}(s')\wedge\pth{s'+p_{f,t}}\wedge 1,
  \]
  where
  \[
    p_{M,t} = \inf\brcb{n\geq 1 : M^{(n)}(t)\text{ exists and }M^{(n)}(t)\neq 0},
  \]
  with the usual convention $\inf\brc{\emptyset} = +\infty$. 
  
  We first prove that almost surely for all $t\in\R_+$, $p_{M,t}\geq 2$.
  Similarly to the proof of Theorem \ref{th:classic_pseudo_frontiers}, up to an enlargement argument, there exists a Brownian motion such that almost surely for all $t\in\R_+$, $M_t = B_{\VQ{M}_t}$. Let fix $\omega\in\Omega$, $t\in\R_+$ and suppose there exists $l_t(\omega)\neq 0$ such that
  \[
    \lim_{h\rightarrow 0} \frac{ M_{t+h} - M_t }{h}(\omega) = l_t(\omega).
  \]
  Without any restriction, we can assume $l_t>0$. Thus, there exists $\rho>0$ such that for all $h\in B(0,\rho)$, 
  \[
    \frac{l_t}{2} \leq \frac{ M_{t+h} - M_t }{h} \leq \frac{3l_t}{2}, \quad\text{ i.e. }\quad M_t+\frac{l_t h}{2} \leq M_{t+h} \leq M_t + \frac{3l_t h}{2}.
  \]
  As a consequence, we obtain,
  \[
    \max_{t-\rho\leq u\leq t} M_u = M_t = \min_{t\leq u\leq t+\rho} M_u.
  \]
  Since $M = B_{\VQ{M}}$ and $\VQ{M}_\cdot$ is continuous and non-decreasing, there exists $\delta>0$ such that
  \[
    \max_{x-\delta\leq u\leq x} B_u = B_x = \min_{x\leq u\leq x+\rho} B_u, \qquad\text{where $x=\VQ{M}_t(\omega)$.}
  \]
  Therefore, $x$ is a point of increase for $B$, as defined in \cite{Dvoretzky.ErdHos.ea(1961)}. Nevertheless, it is proved in \cite{Dvoretzky.ErdHos.ea(1961)} that Brownian sample paths almost surely have no point of increase.
  
  \noindent Hence, almost surely for all $t\in\R_+$, if the limit $l_t$ exists, it is equal to $0$, which proves that $p_{M,t}\geq 2$.\medskip
  
  Then, let $t\in\R$ and $s'\geq-\psd{\alpha}_{M,t}$ such that $\Sigma_{M,t}(s') < \pth{2+s'}$. 
  Then, since $p_{M,t}\geq 2$, we have
  \begin{align*}
    \Sigma_{M,t}(s') 
    = \sigma_{M,t}(s')\wedge\pth{s'+p_{f,t}}\wedge 1
    = \sigma_{M,t}(s')\wedge 1.
  \end{align*}
  Finally, as $\Sigma_{M,t}(s') = \frac{1}{2}\Sigma_{\VQ{M},t}\ptha{{2s'}}$ and $\Sigma_{\VQ{M},t}\leq 1$, we deduce $\Sigma_{M,t}(s') = \sigma_{M,t}(s').$
  
  For all $t\in\R_+$, unless $M$ is locally constant at $t$, we know that $\widetilde{\psd{\alpha}}_{M,t} = \Sigma_{M,t}(0) \leq 1$. Therefore, we get the second equality,
  $
    \widetilde{\alpha}_{M,t} = \widetilde{\psd{\alpha}}_{M,t} = \widetilde{\psd{\alpha}}_{\VQ{M},t} /2.
  $
\end{proof}

To end this section, we establish a lower bound for the pseudo 2-microlocal frontiers of semimartingales.
\begin{proposition} \label{prop:semiMgs}
  Let $X = M + A$ be a continuous semimartingale, where $M$ is a local continuous martingale and $A$ a continuous finite variation process.
  Then, with probability one, for any $t\in\R_+$, the pseudo 2-microlocal frontiers of $X$ satisfies
  \[
    \forall s'\geq-\psd{\alpha}_{X,t};\quad \Sigma_{X,t}(s') \geq \Sigma_{M,t}(s') \wedge \Sigma_{A,t}(s').
  \]
  Furthermore, when $t\in\R_+$ and $s'\geq-\psd{\alpha}_{X,t}$ are such that $X$ satisfies one of these two conditions
  \begin{enumerate}
	  \item $\Sigma_{M,t}(s') \neq \Sigma_{A,t}(s');$
	  \item $A$ is locally monotonic at $t$,
  \end{enumerate}
  then the equality $\Sigma_{X,t}(s') = \Sigma_{M,t}(s') \wedge \Sigma_{A,t}(s')$ holds.
\end{proposition}
\begin{proof}
  Let $t\in\R_+$, $s'\geq-\psd{\alpha}_{X,t}$ and $\sigma < \Sigma_{M,t}(s') \wedge \Sigma_{A,t}(s')$. According to the definition of the pseudo 2-microlocal frontier, we know that $M$ and $A$ belong to $\psd{C}^{\sigma,s'}_{t}$. Therefore, $X = M+A \in \psd{C}^{\sigma,s'}_{t}$, which proves that $\Sigma_{X,t}(s') \geq \sigma$, for all $\sigma < \Sigma_{M,t}(s') \wedge \Sigma_{A,t}(s')$. \vsp
  
  Let now consider the two cases where the equality holds.
  \begin{enumindent}
    \item If $\Sigma_{M,t}(s') > \Sigma_{A,t}(s')$. Let suppose $\Sigma_{X,t}(s') >  \Sigma_{M,t}(s') \wedge \Sigma_{A,t}(s')$, then there exists $\sigma$ such that $\Sigma_{A,t}(s') < \sigma$, $\Sigma_{M,t}(s') > \sigma$ and $\Sigma_{X,t}(s') >\sigma$. Therefore, $M$ and $X$ belong to $\psd{C}^{\sigma,s'}_{t}$, and $A = X-M$ as well. But this is in contradiction with the inequality $\Sigma_{A,t}(s') < \sigma$, and therefore we must have
    \[
      \Sigma_{X,t}(s') =  \Sigma_{M,t}(s') \wedge \Sigma_{A,t}(s').
    \]
    The case $\Sigma_{M,t}(s') < \Sigma_{A,t}(s')$ is treated similarly.
    
    \item Let $\sigma > \Sigma_{M,t}(s') \wedge \Sigma_{A,t}(s')$. To prove the equality, we must find sequences $(s_n)_n$ and $(t_n)_n$ which converges to $t$ and such that 
    \begin{align*}
      \forall n\in\N;\quad\abs{ X_{s_n} - X_{t_n}  } \geq \abs{s_n-t_n}^\sigma \pthb{ \abs{t - s_n}+\abs{t - t_n} }^{-s'},
    \end{align*}
     Without any loss of generality, let suppose $A$ is locally increasing. Then, according to the proof of Theorem \ref{th:mg_vq} (and Lemma \ref{lemma:bm_incr}), we can find such sequences for $M$ and which satisfy for all $n\in\N$, $s_n\leq t_n$ and $M_{s_n} \leq M_{t_n}$. Then, if we consider increments of $X$, we get
    \begin{align*}
      X_{t_n} - X_{s_n} 
      &= M_{t_n} - M_{s_n} + A_{t_n} - A_{s_n} \\
      &\geq M_{t_n} - M_{s_n} \qquad \text{as $A$ is locally increasing,}\\
      &\geq \abs{s_n-t_n}^\sigma \pthb{ \abs{t - s_n}+\abs{t - t_n} }^{-s'}.
    \end{align*}
    This proves that $X\notin\psd{C}^{\sigma,s'}_{t}$ for all $\sigma > \Sigma_{M,t}(s') \wedge \Sigma_{A,t}(s')$.
  \end{enumindent}
\end{proof}


The Brownian motion is a simple martingale which has a deterministic regularity 
\[
  \text{a.s. }\forall t\in\R_+, \forall s'\in\R;\quad \Sigma_{B,t}(s') = \pthbb{ \frac{1}{2}+s' }\wedge\frac{1}{2},
\]
as initially proved \cite{Herbin.Levy-Vehel(2009)} and as confirmed by Theorem \ref{th:mg_vq}.

In the following, we exhibit stochastic processes which have more eccentric regularity. In particular, we construct martingales with a non-deterministic 2-microlocal frontier, which shows that the range of possible behaviours is different than the Gaussian case detailed in \cite{Herbin.Levy-Vehel(2009)}. Quadratic variations constructions are detailed in \ref{sec:app_mg}.

In this first example, we prove that there exist martingales with a non-trivial 2-microlocal frontier, similar to "chirp" regularity.
\begin{example} \label{ex:mg1}
  In Lemma \ref{lemma:function} is constructed a continuous non-decreasing function $g_\alpha$ such that at a given $t_0$,
  \[
    \forall s'\geq -1;\quad \Sigma_{g_\alpha,t_0}(s') = \pthBB{ \frac{s'+ 1}{1-\log_2(\alpha) }}\wedge 1,
  \]
  where $\alpha$ is a parameter in $\ivoo{0,1}$.
 
  Based on this deterministic function, we construct a martingale. Let $\beta$ be a Brownian motion and $U$ a uniform variable on $\ivff{0,1}$ independent of $\beta$. We easily verify that $(g_U(t))_{t\in\R_+}$ is time change for the Brownian motion $\beta$ and therefore, we set the following martingale
  \[
    \forall t\in\R_+\quad M_t = \beta_{g_U(t)}.
  \]
  Then, using Theorem \ref{th:mg_vq} and Corollary \ref{cor:eq_cl_psd}, we obtain the 2-microlocal frontier of $M$: 
  \[
    \forall s'\geq -\tfrac{1}{2};\quad \sigma_{M,t_0}(s') = \frac{1}{2}\Sigma_{\VQ{M},t_0}\pth{{2s'}} = \pthBB{ \frac{s'+\frac{1}{2}}{1-\log_2(U)} }\wedge \frac{1}{2},
  \]
  since $\VQ{M}_t = g_U(t)$. 
  
  Similarly, we can also consider the time change $\pthB{g_{\pthb{2^{-\abs{\beta_{t_0}}}}}(t)}_{t\in\R_+}$. It can be easily checked that it satisfies the necessary hypotheses, and therefore, in this case the 2-microlocal frontier of the martingale is equal to
  \[
    \forall s'\geq-\tfrac{1}{2};\quad \sigma_{M,t_0}(s') = \frac{1}{2}\Sigma_{\VQ{M},t_0}\pth{{2s'}} = \pthbb{ \frac{s'+\frac{1}{2}}{1+\abs{M_{t_0}}} }\wedge \frac{1}{2}.
  \]
\end{example}

Therefore, in contrary to results proved in \cite{Herbin.Levy-Vehel(2009)} in the case of Gaussian processes, at a fixed $t_0$, there exist martingales with a random and non trivial 2-microlocal frontier which can even depend on the values of the martingale itself. We note that stochastic processes with such regularity properties are called self-regulating processes and have already been exhibited in the literature (e.g. \cite{Echelard.Vehel.ea(2010)}).

Finally, we observe that the structure of martingales does not allow to extend this kind of particular regularity to all points on the trajectory. Indeed, a simple consequence of Theorem \ref{th:mg_vq} and the monotonicity of the quadratic variation is that almost surely, for almost all $t\in\R_+$
\[
  \widetilde\alpha_{M,t} = 
  \begin{cases}
    \, +\infty \quad & \text{if $M$ is locally constant at $t$;} \\
    \, \tfrac{1}{2} & \text{otherwise.}
  \end{cases}
\]

\subsection{Time changed multifractional Brownian motion}

Theorem \ref{th:mg_vq} involves the $\tfrac{1}{2}$-H\"older regularity of Brownian motion. It is a natural question to investigate the case of more general processes whose local regularities can be prescribed. Among these, fractional and multifractional Brownian motions, which are natural extensions of Brownian motion, but not martingales.

Both are well-known Gaussian processes, respectively introduced in \cite{Mandelbrot.VanNess(1968)} and \cite{Benassi.Jaffard.ea(1997)}, \cite{Peltier.Levy-Vehel(1995)}. In the article, we will always consider a multifractional Brownian motion (mBm) with regularity function $H:\R\rightarrow\ivff{a,b}\subset\ivoo{0,1}$ which has the following form:
\[
  X_t = \frac{1}{\Gamma\pthb{ H(t) + \tfrac{1}{2} }} \int_\R \bkta{ (t-u)_+^{H(t)-1/2} - (-u)_+^{H(t)-1/2} } \dt W_u,
\]
even if it has been proved in \cite{Stoev.Taqqu(2006)} that the general mBm has a more complex structure. In the case $H = \tfrac{1}{2}$, we obtain a classic Brownian motion. The H\"older regularity of this process has been widely studied in the literature (see e.g. \cite{Peltier.Levy-Vehel(1995)}, \cite{Ayache.Cohen.ea(2000)}, \cite{Ayache.Taqqu(2005)}, \cite{Herbin(2006)} for sample paths regularity and \cite{Meerschaert.Wu.ea(2008)}, \cite{Boufoussi.Dozzi.ea(2007)} for local time properties). Therefore, we know that H\"older exponents of $X$ at $t$ only depend on $H(t)$ and the regularity of $H$ at $t$. 

Later in the article, we will always suppose that the regularity function $H$ satisfies the hypothesis $\Hi_\beta$:
\[
  (\Hi_\beta): \text{$H$ is $\beta$-H\"older continuous with }\sup_{t\in\R} H(t) <\beta.
\]
In that case, let us recall the 2-microlocal frontier of the mBm obtained in \cite{Herbin.Levy-Vehel(2009)}.

\begin{proposition} \label{prop:mbm_2ml}
  Let $X$ be a multifractional Brownian motion whose regularity function $H$ satisfies the hypothesis $\Hi_\beta$. Then, with probability one, for all $t\in\R_+$, classic and pseudo 2-microlocal frontier are equal to
  \[
    \forall s'\in\R;\quad \sigma_{X,t}(s') = \Sigma_{X,t}(s') = \pthb{s'+H(t)}\wedge H(t).
  \]
  Hence, in particular, we have $\alpha_{X,t} = \widetilde\alpha_{X,t} = H(t)$.
\end{proposition}
\begin{proof}
  The proof in \cite{Herbin.Levy-Vehel(2009)} only concerns the classic 2-microlocal frontier with $s'\in\ivff{-H(t),0}$. The equality can be easily extended to every $s'\leq 0$ since we know the 2-microlocal frontier must be concave and have left- and right-derivatives in the interval $\ivff{0,1}$.
  
  Furthermore, using Theorem \ref{th:classic_pseudo_frontiers}, since $X$ is nowhere differentiable, we know that $\Sigma_{X,t} = \sigma_{X,t}\wedge 1$. Finally, if $s'\geq 0$, using Lemma \ref{lemma:pseudo_2ml_4}, we obtain the equality $\Sigma_{X,t}(s') = H(t)$ as the function $H$ is continuous, and also $\sigma_{X,t}(s') = H(t)$ as a consequence.
\end{proof}

We now prove a technical result, related to mBm's increments, which is used to extend the main theorem to a time changed mBm.
\begin{lemma} \label{lemma:mbm_incr}
  Let $X$ be a multifractional Brownian motion, satisfying the hypothesis $\Hi_\beta$. Then, almost surely for all $p\in\N$ and $\eps>0$, there exists $N(\omega)\in\N$ such that:
  \[
    \forall n\geq N(\omega),\ \forall i\in\brc{0,\dotsc,2^{n+p}-1},\ \exists u,v\in\ivff{t_i^{(n)},\ t_{i+1}^{(n)}};\quad \abs{X_u - X_v} \geq \rho_n^{H_i^{(n)}+\eps},
  \]
  where for all $i,n\in\N$, $\rho_n=2^{-n}$, $t_i^{(n)} = i2^{-n}$ and $H_i^{(n)} = \sup_{u\in\ivff{t_i^{(n)},t_{i+1}^{(n)}}} H(u)$.
\end{lemma}
\begin{proof}
  As noted previously, we use the following integral representation of the mBm:
  \begin{align*}
    X_t &= \frac{1}{\Gamma\pthb{ H(t) + \tfrac{1}{2} }}\int_{-\infty}^t \bkta{ (t-u)_+^{H(t)-1/2} - (-u)_+^{H(t)-1/2} } \dt W_u \\
    &\eqdef C(H(t))\int_{-\infty}^t K\pthb{ u, t, H(t) } \,\dt W_u.
  \end{align*}
  Let $\eps>0$ and $p\in\N$. For every $N\in\N$, we consider the following event $A_N$:
  \[
    A_N = \bigcap_{n\geq N} \bigcap_{0 \leq i < 2^{n+p}} \brcB{\exists u,v\in\ivff{t_i^{(n)},\ t_{i+1}^{(n)}} : \abs{X_u - X_v} \geq \rho_n^{H_i^{(n)}+\eps}},
  \]
  We fix $n\geq N$ and $i\in\brc{0,\dotsc, 2^p-1}$. We denote $\Delta t_n = \rho_n^{1+\eps}$ and $m = \lfloor 2^{n\eps}\rfloor$ and we divide the interval $\ivff{t_i^{(n)},t_{i+1}^{(n)}}$ in $m$ smaller intervals $\ivff{u_k,u_{k+1}}$ such that $u_{k+1} = u_k + \Delta t_n$ for all $k\in\brc{0,\dotsc,m}$.
  
  \noindent Let us then consider the events
  \[
    E_i^{(n)} = \bigcap_{k=1}^{m} \brca{ \abs{ X_{u_k} - X_{u_{k-1}} } < \rho_n^{H_i^{(n)}+\eps} }.
  \]
  We estimate an upper bound for $\pr{E_i^{(n)}}$,
  \begin{align*}
    \pr{E_i^{(n)}}
    &= \prbb{ \bigcap_{k=1}^{m} \brca{ \abs{ X_{u_k} - X_{u_{k-1}} } < \rho_n^{H_i^{(n)}+\eps} } } \\
    &= \espbb{ \prod_{k=1}^{m-1} \indi_{ \brcb{ \abs{ X_{u_k} - X_{u_{k-1}} } < \rho_n^{H_i^{(n)}+\eps} } } \espcB{ \indi_{ \brcb{ \abs{ X_{u_m} - X_{u_{m-1}} } < \rho_n^{H_i^{(n)}+\eps}  } } }{W_x , x\leq u_{m-1} } }
  \end{align*}
  Let consider the last term
  \begin{align*}
    &\prcb{ \abs{ X_{u_m} - X_{u_{m-1}} } < \rho_n^{H_i^{(n)}+\eps} }{W_x , x\leq u_{m-1} } \\
    &= \prcB{ \absb{ C\pthb{H(u_m)} \int_{u_{m-1}}^{u_m} (u_m - u)^{H(u_m)-1/2} \dt W_u + Y_m } < \rho_n^{H_i^{(n)}+\eps} }{W_x , x\leq u_{m-1} },
  \end{align*}
  where 
%
  $
    Y_m = \int_{-\infty}^{u_{m-1}} \bktb{ C\pthb{H(u_m)} K\pthb{ u, u_m, H(u_m) } - C\pthb{H(u_{m-1})} K\pthb{ u, u_{m-1}, H(u_{m-1}) } } \dt W_u.
  $
  
  We note that $Y_m$ is $\sigma\brca{W_x , x\leq u_{m-1}}$-measurable, whereas the stochastic integral on the interval $\ivff{u_{m-1},u_m}$ is independent of $\sigma\brca{W_x , x\leq u_{m-1}}$.
  Furthermore, this last term is a centered Gaussian random variable with the following variance:
  \begin{align*}
    \sigma_m^2
    = C\pthb{H(u_m)}^2 \int_{u_{m-1}}^{u_m} (u_m-u)^{2H(u_m)-1} \dt u 
    = \frac{ C\pthb{H(u_m)}^2 }{ 2H(u_m) } \pthb{\Delta t_n}^{2H(u_m)}.
  \end{align*}
  Therefore, we obtain the inequality
  \begin{align*}
    \prcb{ \abs{ X_{u_m} - X_{u_{m-1}} } < \rho_n^{H_i^{(n)}+\eps} }{W_x , x\leq u_{m-1} }
    &= \frac{1}{\sigma_m\sqrt{2\pi}} \int_{-\rho_n^{H_i^{(n)}+\eps}}^{\rho_n^{H_i^{(n)}+\eps}} \exp\pthbb{-\frac{(u -Y_m )^2}{2\sigma_m^2}} \dt u \\
    &\leq \frac{2}{\sqrt{2\pi}} \frac{\rho_n^{H_i^{(n)}+\eps}}{\sigma_m}. \\
  \end{align*}
  Using previous expressions, we get (where $K$ is positive constant)
  \begin{align*}
    \frac{2}{\sqrt{2\pi}} \frac{\rho_n^{H_i^{(n)}+\eps}}{\sigma_m}
    &= \frac{ 2\sqrt{H(u_m)} }{ \sqrt{\pi}C\pthb{H(u_m)} } \frac{\rho_n^{H_i^{(n)}+\eps}}{\Delta t_n^{H(u_m)}} \\
    &= \frac{ 2\sqrt{H(u_m)} }{ \sqrt{\pi}C\pthb{H(u_m)} } \rho_n^{H_i^{(n)}-H(u_m)+\eps(1-H(u_m))} 
    \leq 2^{-n\eps(1-b) + K},
  \end{align*}
  as $H$ and $C$ are continuous, $H(u_m)\leq H_i^{(n)}$ and $H(t)\in\ivff{a,b}$ for all $t\in\R_+$. 
  Thus, by induction on $k\in\brc{1,\dotsc,m}$, we have
  \begin{align*}
    \pr{E_i^{(n)}}
    = \prbb{ \bigcap_{k=1}^{m} \brca{ \abs{ X_{u_k} - X_{u_{k-1}} } < \rho_n^{H_i^{(n)}+\eps} } } 
    \leq 2^{(-n\eps(1-b) + K)m} 
    \leq 2^{(-n\eps(1-b) + K)(2^{n\eps}-1)}.
  \end{align*}
  Therefore, we obtain:
  \begin{align*}
    &\prb{\forall u,v\in\ivff{t_i^{(n)},t_i^{(n+1)}} : \abs{X_u - X_v} < \rho_n^{H_i^{(n)}+\eps}} \\
    &\quad\leq \prbb{ \bigcap_{k=1}^{m} \brca{ \abs{ X_{u_k} - X_{u_{k-1}} } < \rho_n^{H_i^{(n)}+\eps} } }
    \leq 2^{(-n\eps(1-b) + K)(2^{n\eps}-1)}.
  \end{align*}
  Finally, if we consider the event $A_N$:
  \begin{align*}
    \pr{A_N^c}
    &= \prbb{ \bigcup_{n\geq N} \bigcup_{0 \leq i < 2^{n+p}} \brcB{\forall u,v\in\ivff{t_i^{(n)},t_i^{(n+1)}} : \abs{X_u - X_v} < \rho_n^{H_i^{(n)}+\eps}} } \\
    &\leq \sum_{n\geq N} \sum_{0 \leq i < 2^{n+p}} \prb{\forall u,v\in\ivff{t_i^{(n)},t_i^{(n+1)}} : \abs{X_u - X_v} < \rho_n^{H_i^{(n)}+\eps}} \\
    &\leq \sum_{n\geq N} 2^{(-n\eps(1-b) + K)(2^{n\eps}-1)+n+p} 
    \leq \widetilde K 2^{-N}.
  \end{align*} 
  Therefore, using Borel-Cantelli Lemma on $(A_N^c)_{N\in\N}$, we prove the expected result.
\end{proof}

\begin{proposition} \label{prop:mbm_incr}
  Let $X$ be a multifractional Brownian motion, satisfying the hypothesis $\Hi_\beta$. Then, almost surely, for any $T\geq 0$ and all $\eps>0$, there exists $h(\omega)>0$ such that for all $\rho\leq h(\omega)$ and $t\in\ivff{0,T}$, we have:
  \[
    \sup_{u,v\in B(t,\rho)} \brcB{ \abs{X_u - X_v} } \geq \rho^{H(t)+\eps}.
  \]
\end{proposition}
\begin{proof}
  Let $T>0$ and $\eps>0$, there exists $p\in\N$ such that $T \leq 2^p$. Using notations from Lemma \ref{lemma:mbm_incr}, we consider $h(\omega) = 2^{-N(\omega)}$. Then, for every $\rho\leq h(\omega)$ and $t\in\ivff{0,T}$, there exist $n,k\in\N$ such that $2^{-(n+1)} \leq \rho \leq 2^{-n}$ and $\ivff{k 2^{-(n+1)},(k+1) 2^{-(n+1)}}\subseteq\ivff{t-\rho,t+\rho}$.
  
  Then, from Lemma \ref{lemma:mbm_incr}, there exists $u,v\in\ivff{k 2^{-(n+1)},(k+1) 2^{-(n+1)}}$ such that
  \[
    \abs{X_u - X_v} \geq 2^{-(n+1)\pth{ H_{k,n+1}+2\eps }}.
  \]
  As $H$ is uniformly continuous on the interval $\ivff{0,T}$, $h(\omega)$ can be chosen small enough such that for all $n,k\in\N$, $n\geq N$, $H_{k,n}-H(t) \leq \eps$ for every $t\in\ivff{k2^{-n},(k+1)2^{-n}}$.
  
  \noindent Therefore, we have
  \[
    \abs{X_u - X_v} \geq K 2^{-n\pth{ H_{k,n+1}+3\eps }} \geq K \rho^{ H(t)+\eps },
  \]
  which proves the result.
\end{proof}

\begin{remark}
  Close results from Proposition \ref{prop:mbm_incr} have been previously obtained in the literature, usually based on the study of local times. Thereby, inequality $(8.8.26)$ in \cite{Adler(1981)} is related to our result in the particular case of fractional Brownian motion.
  
  Furthermore, Theorem 3.6 in \cite{Ayache.Shieh.ea(2011)} states that for all $\delta>0$, $\eps>0$, 
  \[
    \liminf_{\rho\rightarrow 0} \inf_{t\in\ivff{\delta,1}} \sup_{u\in B(t,\rho)} \frac{ \abs{X_t - X_u} }{ \rho^{\overline H+\eps} } > 0\quad\text{a.s.},
  \]
  where $\overline H = \max_{t\in\ivff{\delta,1}} H(t)$. This last property is slightly weaker than our result since Proposition \ref{prop:mbm_incr} is equivalent to 
  \[
    \liminf_{\rho\rightarrow 0} \inf_{t\in\ivff{0,1}} \sup_{u\in B(t,\rho)} \frac{ \abs{X_t - X_u} }{ \rho^{H(t)+\eps} } > 0 \quad\text{a.s. for all $\eps>0$.}
  \]
\end{remark}

We finally present the main result of this part which gives the classic 2-microlocal frontier of a time changed multifractional Brownian motion. As the proof is not modified compare to the martingale specific case, we only recall the major steps.  
The proof of the first side inequality does not change at all, we only use the uniform regularity of the mBm recalled previously in Proposition \ref{prop:mbm_2ml}. The converse inequality is also shown similarly, using Proposition \ref{prop:mbm_incr} proved below.
\begin{theorem} \label{th:mbm_comp}
  Let $X$ be a multifractional Brownian motion, satisfying $\Hi_\beta$ and $U$ be a continuous positive process. We denote $Z$ the compound process:
  \[
    \forall t\in\R_+;\quad Z_t = X_{U_t}.
  \]
  Then, with probability one, for all $t\in\R_+$, the pseudo 2-microlocal frontier of $Z$ at $t$ verifies the following equality:
  \[
    \forall s'\geq -\alpha_{Z,t};\quad \Sigma_{Z,t}(s') = H(U_t) \cdot \Sigma_{U,t}\ptha{\frac{s'}{H(U_t)}}.
  \]
  Consequently, with probability one, pointwise and local H\"older exponents satisfy:
  \[
    \forall t\in\R_+; \quad\alpha_{Z,t} = \psd{\alpha}_{Z,t} = H(U_t)\cdot\psd{\alpha}_{U,t} \quad\text{and}\quad \widetilde\alpha_{Z,t} = \widetilde{\psd{\alpha}}_{Z,t} = H(U_t)\cdot\widetilde{\psd{\alpha}}_{U,t}.
  \]
\end{theorem}

\begin{remark}
  Properties of time changed fractional Brownian motions have already been studied in the literature. It first appeared in \cite{Mandelbrot(1997)} because of its particular multifractal properties. The multifractal spectra is obtained in \cite{Riedi(2003)} using similar technics to the ones used in Theorem \ref{th:mbm_comp}.
\end{remark}


\section{2-microlocal frontier of stochastic integrals} \label{sec:2ml_int}

As stated in the introduction section, we want to extend the characterization of the Wiener integral regularity obtained in \cite{Herbin.Levy-Vehel(2009)}. Let first recall this result. We define the stochastic process $X$ as the Wiener integral 
$
  X_t = \int_0^t \eta(u) \dt W_u,
$
where $\eta$ is an $L^2$-deterministic function. Then, according to Theorem 4.12 in \cite{Herbin.Levy-Vehel(2009)}, for all $t_0\in\R_+$, the pseudo 2-microlocal frontier of $X$ at $t_0$ is almost surely given by 
\[
  \forall s'\in\ivff{-\psd{\alpha}_{X,t},0};\quad \Sigma_{X,t_0}(s') = \frac{1}{2} \Sigma_{\int_0^\sbullet \eta^2(u) \dt u,t_0}(2s').
\]

In this section, we will consider the extension of this equality to stochastic integrals
$
  X_t = \int_{0}^t H_u \dt M_u,
$
where $M$ is a continuous local martingale and $H$ is a continuous progressive stochastic process. Here, the process $H$ is supposed to be continuous to avoid technical problems on the definition of its 2-microlocal frontier, but this hypothesis could be weakened with still the same results if one wants to consider a larger class of integrands.

We begin by expressing an equality which a straight forward consequence of our previous analysis of martingales regularity.
\begin{theorem} \label{th:stoc_int_ext}
  Let $\brc{X_t; t\in\R_+}$ be defined by the stochastic integral
  \[
     \forall t\in\R_+;\quad X_t = \int_{0}^t H_u \dt M_u,
  \]
  where $M$ is a continuous local martingale and $H$ is a continuous progressive stochastic process. Then, with probability one, for all $t\in\R_+$, the pseudo 2-microlocal frontier of $X$ at $t$ is equal to
  \[
    \forall s'\geq -\psd{\alpha}_{X,t};\quad \Sigma_{X,t}(s') = \frac{1}{2}\Sigma_{\int_{0}^\sbullet H_u^2 \dt \VQ{M}_u,t}\ptha{{2s'}}.
  \]
\end{theorem}
\begin{proof}
  This result is a simple application of Theorem \ref{th:mg_vq}, using the well-known equality
  \[
    \forall t\in\R_+;\quad\VQ{\int_{0}^\sbullet H_u \dt M_u}_t = \int_{0}^t H_u^2 \dt \VQ{M}_u.
  \]
\end{proof}

Theorem \ref{th:stoc_int_ext} is obviously an improvement of the result stated on the Wiener integral. First, we obtain an equality for any stochastic integral with respect to a continuous local martingale. Furthermore, we extend the equality to every $s'$ positive. Finally, a more subtle enhancement lies in the fact that the Theorem \ref{th:stoc_int_ext} gives a uniform almost sure result, i.e. "almost surely for all $t\in\R_+$", whereas the original one is not. We present at the end of this section an example where the difference between these two kind of results appears.\vspar

The remaining of this section is devoted to a characterization of the pseudo 2-microlocal frontier of stochastic integral using both regularities of the martingale $M$ and the integrand $H$. We begin a technical lemma related to the quadratic variation regularity.
\begin{lemma} \label{lemma:stoch_VQ}
  The pseudo 2-microlocal frontier of the increasing process $A = \int_{0}^\sbullet H_s^2 \dt \VQ{M}_s$ satisfies, for all $\omega\in\Omega$ and $t\in\R_+$, 
  \begin{enumerate}
  	\item if $H_t(\omega) \neq 0$,
  	\[
  	  \forall s'\geq-\psd{\alpha}_{A,t};\quad \Sigma_{A,t}(s') = \Sigma_{\VQ{M},t}(s');
  	\]
  	\item if $H_t(\omega) = 0$,
  	\[
  	  \forall s'\geq-\psd{\alpha}_{A,t};\quad \Sigma_{A,t}(s') \geq \Sigma_{\VQ{M},t}(s' + 2\psd{\alpha}_{H,t}),
  	\]
  	where $\psd{\alpha}_{H,t}$ is the pseudo pointwise exponent of $H$ at $t$.
  \end{enumerate}
\end{lemma}
\begin{proof}
  \hfill
  \begin{enumindent}
    \item Since $H$ is continuous and $H_t \neq 0$, there exist $C_1, C_2, \rho > 0$ such that for all $u\in B(t,\rho)$, we have $C_1\leq \abs{H_u}^2\leq C_2$. 
    
    Therefore, we observe that for all $u \leq v\in B(t,\rho)$,
    \[
      C_1 \abs{\VQ{M}_v - \VQ{M}_u} \leq \absa{ \int_0^v H_s^2 \dt\VQ{M}_s - \int_0^u H_s^2 \dt\VQ{M}_s } \leq C_2 \abs{\VQ{M}_v - \VQ{M}_u},
    \]
    Based on the definition of the pseudo 2-microlocal frontier, these two inequalities prove the first point.
  
    \item When $H_t=0$, we observe that for all $s'\geq-\psd{\alpha}_{\VQ{M},t}$ and $\eps>0$, there exists $C>0$ and $\rho>0$ such that for all $u\leq v\in B(t,\rho)$,
    \begin{align*}
      \absa{ \int_0^v H_s^2 \dt\VQ{M}_s - \int_0^u H_s^2 \dt\VQ{M}_s } 
      &= \absa{ \int_u^v (H_s - H_t)^2 \dt\VQ{M}_s } \\
      &\leq C \int_u^v \abs{s - t}^{2\psd{\alpha}_{H,t}-\eps} \dt\VQ{M}_s \\
      &\leq C \ptha{\VQ{M}_v - \VQ{M_u}} \pthb{\abs{u - t}+\abs{v - t}}^{2\psd{\alpha}_{H,t}-\eps} \\
      &\leq C \abs{u-v}^{\Sigma_{\VQ{M},t}(s')-\eps} \pthb{\abs{u - t}+\abs{v - t}}^{-s'+2\psd{\alpha}_{H,t}-\eps}
    \end{align*}
    which proves that
    \[
      \Sigma_{A,t}(s'-2\psd{\alpha}_{H,t}+\eps) \geq \Sigma_{\VQ{M},t}(s')-\eps,
    \]
    Hence, when $\eps\rightarrow 0$, the expected inequality is obtained using the continuity of the pseudo 2-microlocal frontier.
  \end{enumindent}
\end{proof}

We can now derive a lower bound for the pseudo 2-microlocal frontier of a stochastic integrals with respect to a continuous local martingale.
\begin{theorem} \label{th:2ml_stoch_int}
  Let $\brc{X_t; t\in\R_+}$ be defined by the stochastic integral
  \[
    \forall t\in\R_+;\quad X_t = \int_{0}^t H_u \dt M_u,
  \]
  where $M$ is a continuous local martingale and $H$ is a continuous progressive stochastic process.
  
  Then, there exists an event $\Omega_0$ such that $\pr{\Omega_0}=1$ and, for all $\omega\in\Omega_0$ and $t\in\R_+$, the pseudo 2-microlocal frontier of the stochastic integral $X$ satisfies
  \begin{enumerate}
  	\item if $H_t(\omega) \neq 0$,
  	\[
  	  \forall s'\geq-\psd{\alpha}_{X,t};\quad \Sigma_{X,t}(s') = \Sigma_{M,t}(s');
  	\]
  	\item if $H_t(\omega) = 0$,
  	\[
  	  \forall s'\geq-\psd{\alpha}_{X,t};\quad \Sigma_{X,t}(s') \geq \Sigma_{M,t}(s' + \psd{\alpha}_{H,t}),
  	\]
  	where $\psd{\alpha}_{H,t}$ is the pseudo pointwise exponent of $H$ at $t$.
  \end{enumerate}
\end{theorem}
\begin{proof}
  \hfill
  \begin{enumindent}
    \item When $H_t(\omega) \neq 0$, Lemma \ref{lemma:stoch_VQ} leads to 
    \[
  	  \forall s'\geq-\tfrac{ \psd{\alpha}_{\VQ{M},t} }{2};\quad \frac{1}{2}\Sigma_{\int_{0}^\sbullet H_u^2 \dt \VQ{M}_u,t}(2s') = \frac{1}{2}\Sigma_{\VQ{M},t}(2s').
  	\]
    Hence, the application of Theorem \ref{th:mg_vq} allows to obtain the expected equality.
    
    \item In the other case, $H_t(\omega) \neq 0$, the same Lemma \ref{lemma:stoch_VQ} implies
  	\[
  	  \forall s'\geq-\tfrac{ \psd{\alpha}_{\int_{0}^\sbullet H_u^2 \dt \VQ{M}_u,t} }{2};\quad \frac{1}{2}\Sigma_{\int_{0}^\sbullet H_u^2 \dt \VQ{M}_u,t}(2s') \geq \frac{1}{2}\Sigma_{\VQ{M},t}(2s' + 2\psd{\alpha}_{H,t}),
  	\]
  	which also induces the expected inequality, using Theorem \ref{th:mg_vq}.
  \end{enumindent}
\end{proof}

Theorem \ref{th:2ml_stoch_int} can be improved in the case of a stochastic integral with respect to Brownian motion since in this case, the integral with respect to the quadratic variation is reduced to a classic Lebesgue integral.
\begin{theorem} \label{th:2ml_stoch_int_bm}
  Let $\brc{X_t; t\in\R_+}$ be defined by the stochastic integral
  \[
    \forall t\in\R_+;\quad X_t = \int_{0}^t H_u \dt B_u,
  \]
  where $B$ is a Brownian motion and $H$ is a continuous progressive stochastic process.
  
    Then, there exists an event $\Omega_0$ such that $\pr{\Omega_0}=1$ and for all $\omega\in\Omega_0$ and $t\in\R_+$, the pseudo 2-microlocal frontier of the stochastic integral $X$ satisfies
  \begin{enumerate}
  	\item if $H_t(\omega) \neq 0$,
  	\[
  	  \forall s'\in\R;\quad \Sigma_{X,t}(s') = \Sigma_{B,t}(s') = \ptha{\frac{1}{2} + s'}\wedge\frac{1}{2};
  	\]
  	\item if $H_t(\omega) = 0$,
  	\[
  	  \forall s'\geq -\psd{\alpha}_{X,t};\quad \Sigma_{X,t}(s') = \ptha{\frac{1}{2} + \frac{\Sigma_{H^2,t}(2s')}{2}}\wedge\frac{1}{2},
  	\]
  	unless $H$ is locally equal to zero at $t$, which induces in that case: $\Sigma_{X,t}=+\infty$.
  \end{enumerate}
\end{theorem}
\begin{proof}
  \hfill
  \begin{enumindent}
    \item To obtain the first equality, we use the previous Theorem \ref{th:2ml_stoch_int}, since we know that almost surely for all $t\in\R_+$,
  	\[
  	  \forall s'\in\R;\quad \Sigma_{B,t}(s') = \ptha{\frac{1}{2} + s'}\wedge\frac{1}{2}.
  	\]
  	The formula is extended to all $s'\leq-\tfrac{1}{2}$ using properties of the pseudo 2-microlocal frontier (concavity and derivatives in \ivff{0,1}) proved in Corollary \ref{cor:pseudo_properties}.
  	
  	\item When $H_t(\omega) = 0$, we observe that if $H$ is locally equal to zero, we obtain $\Sigma_{X,t}=+\infty$.
  	
  	In the opposite case, from Theorem \ref{th:classic_pseudo_frontiers}, the pseudo 2-microlocal frontier of the integral satisfies
  	\[
      \forall s'\in\R;\quad \Sigma_{\int_0^\sbullet H_s^2 \dt s,t}(s') = \pthb{ 1 + \Sigma_{H^2,t}(s') }\wedge 1.
    \]
    Hence, Theorem \ref{th:mg_vq} yields the expected formula.
  \end{enumindent}
\end{proof}

\begin{remark}
  The second points in Theorems \ref{th:2ml_stoch_int} and \ref{th:2ml_stoch_int_bm} do not contradict. Indeed, we know that for all $t\in\R_+$ such that $H_t(\omega)=0$,
  \[
    \forall s'\in\R; \quad \Sigma_{\VQ{B},t}(s' + 2\psd{\alpha}_{H,t}) = \pthb{1 + s' + 2\psd{\alpha}_{H,t}}\wedge 1.
  \]
  Furthermore, for all $s'\leq -2\psd{\alpha}_{H,t}$, still using properties of the pseudo 2-microlocal frontier from Corollary \ref{cor:pseudo_properties}, we obtain
  \[
    \Sigma_{H^2,t}(s') \geq s ' + 2\psd{\alpha}_{H,t},
  \]
  which proves the expected inequality for all $s'\in\R$
  \[
    \Sigma_{X,t}(s') = \ptha{\frac{1}{2} + \frac{\Sigma_{H^2,t}(2s')}{2}}\wedge\frac{1}{2} \ \geq\  \ptha{\frac{1}{2} + s' + \psd{\alpha}_{H,t}}\wedge \frac{1}{2} = \Sigma_{B,t}(s' + \psd{\alpha}_{H,t}).
  \]
\end{remark} 
Finally, we also notice that equalities presented in Theorems \ref{th:2ml_stoch_int} and \ref{th:2ml_stoch_int_bm} can be extended to the classic 2-microlocal frontier when the pseudo frontier satisfies Corollary \ref{cor:eq_cl_psd}. \vspar

As an application, we consider two examples of stochastic integral whose 2-microlocal frontier can be completely determined.

\begin{example}
  Let consider an Ornstein-Uhlenbeck process $X$. It admits the following representation:
  \[
    \forall t\in\R_+;\quad X_t = X_0e^{-\theta t} + \mu\pthb{ 1 - e^{-\theta t} } + \int_0^t \sigma e^{\theta(s-t)} \dt B_s,
  \]
  where $\theta>0$, $\mu$ and $\sigma>0$ are parameters and $B$ is a Brownian motion.
  
  Then, almost surely for all $t\in\R_+$, classic and pseudo 2-microlocal frontiers of $X$ are equal to
  \[
    \forall s'\in\R; \quad \sigma_{X,t}(s') = \Sigma_{X,t}(s') = \pthbb{ \frac{1}{2} + s'}\wedge\frac{1}{2},
  \]
  and in particular $\alpha_{X,t} = \widetilde\alpha_{X,t} = \frac{1}{2}$. \vspar
  
 This result is a consequence of Theorem \ref{th:2ml_stoch_int_bm} and the strict positivity of the exponential function. Therefore, the stochastic integral $\int_0^t \sigma e^{\theta s} \dt B_s$ has almost surely the following pseudo 2-microlocal frontier, for all $t \in\R_+$:
  \[
    \forall s'\in\R;\quad \Sigma_{\int_0^\sbullet \sigma e^{\theta s} \dt B_s,t}(s') = \pthbb{ \frac{1}{2} + s'}\wedge\frac{1}{2}.
  \]
  Furthermore, as $t\mapsto e^{-\theta t}$ is a positive $C^\infty$ function, we can easily see that other terms in the expression do not modify the regularity. Finally, Corollary \ref{cor:eq_cl_psd} extends the equality to the classic 2-microlocal frontier.
\end{example}

\begin{remark}
  As the Ornstein-Uhlenbeck process is a Gaussian process, the 2-microlocal frontier can be directly obtained using technics and results introduced in \cite{Herbin.Levy-Vehel(2009)}. Another way to get this result is to note that Lamperti transform of Brownian motion does not modify the regularity.
\end{remark}

The second example illustrates the impact of integrand's regularity on the 2-microlocal frontier of the integral.
\begin{example} \label{ex:mbm_int}
  Let $B$ be a Brownian motion and $X$ a multifractional Brownian motion adapted to $B$ filtration and which satisfies the hypothesis $\Hi_\beta$. We consider the following stochastic integral:
  \[
    \forall t\in\R_+;\quad Z_t = \int_0^t X_s \dt B_s.
  \]
  Then, there exists an event $\Omega_0$ such that $\pr{\Omega_0}=1$ and for all $\omega\in\Omega_0$ and $t\in\R_+$, classic and pseudo 2-microlocal frontier of this integral satisfy:
  \begin{enumerate}
  \item if $X_t(\omega) \neq 0$,
  	\[
  	  \forall s'\in\R;\quad \sigma_{Z,t}(s') = \Sigma_{Z,t}(s') = \pthbb{ \frac{1}{2} + s' }\wedge \frac{1}{2},
  	\]
    and in particular $\alpha_{Z,t} = \widetilde\alpha_{Z,t} = \frac{1}{2}$;
  	\item if $X_t(\omega) = 0$,
  	\[
  	  \forall s'\in\R;\quad \sigma_{Z,t}(s') = \Sigma_{Z,t}(s') = \pthbb{ \frac{1}{2} + H(t) + s' }\wedge \frac{1}{2},
  	\]
  	and in particular $\alpha_{X,t} = \frac{1}{2} + H(t)$ and $\widetilde\alpha_{X,t} = \frac{1}{2}$.
  \end{enumerate}
%
%
  
  Before proving these results, we first note that the stochastic integral is well-defined, as for all $t\in\R_+$:
  \[
    \espbb{\int_0^t (X_s)^2 \dt s } = \int_0^t C(H(s)) s^{2H(s)} \dt s < \infty.
  \]
  According to Theorem \ref{th:2ml_stoch_int_bm}, we only have to characterize the pseudo 2-microlocal frontier $\Sigma_{(X)^2,t}$ in the case $X_t(\omega) = 0$. 
  \begin{enumindent}
    \item
    For the lower bound, the frontier of $X$ is known to be equal to
    \[
      \forall s'\in\R;\quad \Sigma_{X,t}(s') = \pthb{ H(t) + s' }\wedge H(t).
    \]
    Therefore, for all $s'\geq -H(t)$ and $\eps>0$, there exist $\rho>0$ and $C>0$ such that for all $u,v\in B(t,\rho)$,
    \begin{align*}
      \abs{X^2_u-X^2_v} 
      &= \abs{X_u-X_v} \cdot \abs{X_u+X_v} \\
      &\leq C \abs{u-v}^{\pth{ H(t)+s' }\wedge H(t) -\eps} \pthb{\abs{t-u}+\abs{t-v}}^{-s'} \cdot \pthb{\abs{t-u}+\abs{t-v}}^{H(t)-\eps} \\
      &= C \abs{u-v}^{\pth{ 2H(t)+\tilde s' }\wedge H(t) -\eps} \pthb{\abs{t-u}+\abs{t-v}}^{-\tilde s'-\eps},
    \end{align*}
    where $\tilde s'=s'-H(t) \geq -2H(t)$. Hence, it proves that for all $s'\geq -2H(t)$,
    \[
      \forall s'\geq -2H(t);\quad \Sigma_{(X)^2,t}(s') \geq \pthb{ 2 H(t) + s'}\wedge H(t).
    \]
    
    \item To obtain the upper bound, we note that Lemma \ref{lemma:pseudo_2ml_4} implies that $\Sigma_{(X)^2,t} \leq H(t)$.
    
    Then, for all $\eps>0$, there exists a sequence $(u_n)_n$ such that for all $n\in\N$, $\abs{X_{u_n} - X_t} \geq \abs{u_n-t}^{H(t)+\eps}$. Thus, for all $n\in\N$, we also obtain $\abs{X^2_{u_n} - X^2_t} = \abs{X_{u_n} - X_t}^2 \geq \abs{u_n-t}^{2H(t)+2\eps}$, which proves that 
    \[
      \forall s'\geq-2H(t);\quad \Sigma_{(X)^2,t}(s') \leq 2 H(t)+s'.
    \]
  \end{enumindent}
  Therefore, we have the equality
  \[
    \forall s'\in\R;\quad \Sigma_{(X)^2,t}(s') = \pthb{ 2 H(t) + s'}\wedge H(t), 
  \]
  where the extension to all $s'\leq -2H(t)$ is still obtained using properties of the pseudo 2-microlocal frontier (concavity and derivatives between $0$ and $1$).
  
  \noindent Hence, for all $s'\geq-\psd{\alpha}_{X,t}$, we get
  \[
    \Sigma_{Z,t}(s') = \ptha{\frac{1}{2} + \frac{\Sigma_{X^2,t}(2s')}{2}}\wedge\frac{1}{2} = \pthbb{ \frac{1}{2} + H(t) + s' }\wedge \frac{1}{2}.
  \]
  The same properties allow to extend this equality to all $s'\in\R$. Finally, Corollary \ref{cor:eq_cl_psd} is applied to get the 2-microlocal frontier.
\end{example}

%

\begin{remark}
  Example \ref{ex:mbm_int} also shows that there is a main difference between uniform results "a.s. $\forall t\in\R_+$" and simpler ones "$\forall t\in\R_+$ a.s.". Indeed, we know that for all $t\in\R_+$, almost surely $X_t\neq 0$. Therefore, for all $t\in\R_+$, the 2-microlocal frontier of the stochastic integral is almost surely equal to
  \[
  	  \forall s'\in\R;\quad \sigma_{\int_{0}^\sbullet X_s \dt B_s,t}(s') = \pthbb{ \frac{1}{2}+s' }\wedge \frac{1}{2},
  \]
  which is obviously a less precise characterization of the regularity.
\end{remark}


\section{Application to stochastic differential equations} \label{sec:2ml_SDE}

Stochastic differential equations (SDE) are often used as a mathematical representation of natural phenomenons. Local regularity constitutes an interesting indicator to verify how "correctly" a model fits observable data. In this section, we present a method to obtain the 2-microlocal frontier (or at least a lower bound for the local regularity) of SDE solutions.

To simplify our statement and the expressions, we focus on homogeneous diffusions which have the following form:
\[
  \dt X_t = a(X_t) \,\dt B_t + b(X_t) \,\dt t,
\]
where $a$ and $b$ are continuous functions with respective pseudo pointwise H\"older exponents $\psd{\alpha}_{a,x}$ and $\psd{\alpha}_{b,x}$ at $x\in\R$.

We assume the existence of a solution $X$ to this SDE. From Theorem \ref{th:2ml_stoch_int_bm}, there exists an event $\Omega_0$ such that $\pr{\Omega_0} = 1$ and for all $\omega\in\Omega_0$ and $t\in\R_+$, the pseudo 2-microlocal frontier of $X$ and thus the local regularity is completely characterized. Similarly, we have to distinguish different cases, as illustrated in equations \eqref{eq:sde1}, \eqref{eq:sde2}, \eqref{eq:sde3} and \eqref{eq:sde4} obtained on the pseudo 2-microlocal frontier $\Sigma_{X,t}$.

\begin{enumindent}
	\item If $a(X_t(\omega)) \neq 0$, Theorem \ref{th:2ml_stoch_int_bm} directly implies
    \begin{equation} \label{eq:sde1}
  	  \forall s'\in\R;\quad \Sigma_{X,t}(s') = \pthbb{\frac{1}{2} + s'}\wedge\frac{1}{2}.
  	\end{equation}
  \item If $a(X_t(\omega)) = 0$, but is not locally equal to zero, and $b(X_t(\omega))\neq 0$, the stochastic integral $\int_0^\sbullet a(X_s)\dt B_s$ satisfies
    \[
      \forall s'\in\R;\quad \Sigma_{\int_0^\sbullet a(X_s)\dt B_s,t}(s') = \pthbb{\frac{1}{2} + \frac{\Sigma_{a^2(X_\sbullet),t}(2s')}{2}}\wedge\frac{1}{2}.
    \]
    Furthermore, using Theorem \ref{th:int_pseudo_frontier}, as $b(X_t(\omega))\neq 0$ and $s\mapsto b(X_s(\omega))$ is continuous, we have 
    \[
      \forall s'\in\R;\quad \Sigma_{\int_0^\sbullet b(X_s)\dt s,t}(s') = \pth{1 + s'}\wedge 1.
    \]
    These two expressions lead to the following
    \[
      \forall s'\in\R;\quad \Sigma_{X,t}(s') = \pthbb{\frac{1}{2} + \frac{\Sigma_{a^2(X_\sbullet),t}(2s')}{2}}\wedge\pth{1 + s'}\wedge\frac{1}{2}
    \]
    and, in particular
    \[
      \forall s'\in\R;\quad \Sigma_{X,t}(s') \leq \pthb{1 + s'}\wedge\frac{1}{2} \quad\text{and}\quad \sigma_{X,t}(s') \leq \frac{1}{2}.
    \]
    We can easily check that $\psd{\alpha}_{a^2(X_\sbullet),t} \geq 2\psd{\alpha}_{a,X_t} \psd{\alpha}_{X,t}$. Thus, we get the following inequality for the negative part of the pseudo frontier:
    \[
      \forall s'\leq -\psd{\alpha}_{a^2(X_\sbullet),t};\quad \Sigma_{a^2(X_\sbullet),t}(s') \geq 2\psd{\alpha}_{a,X_t} \psd{\alpha}_{X,t} + s'.
    \]
    Using the previous inequalities, we obtain:
    \[
      \forall s'\in\R;\quad \Sigma_{X,t}(s') \geq \pthbb{\frac{1}{2} + \psd{\alpha}_{a,X_t} \psd{\alpha}_{X,t} + s'}\wedge\pth{1 + s'}\wedge\frac{1}{2},
    \]
    and, in particular
    \[
      \psd{\alpha}_{X,t} \geq \pthbb{\frac{1}{2} + \psd{\alpha}_{a,X_t} \psd{\alpha}_{X,t}}\wedge 1.
    \]
    Therefore, we have $\psd{\alpha}_{X,t} \geq \pthB{\frac{1}{2(1 - \psd{\alpha}_{a,X_t})_+}}\wedge 1$, which leads to the following estimation of the pseudo frontier of $X$
    \begin{equation} \label{eq:sde2}
       \forall s'\in\R;\quad \pthbb{\frac{1}{2(1 - \psd{\alpha}_{a,X_t})_+} + s'}\wedge\pthb{1 + s'}\wedge\frac{1}{2} \leq \Sigma_{X,t}(s') \leq \pth{1 + s'}\wedge\frac{1}{2}.
    \end{equation}
    In particular, we note that when $\psd{\alpha}_{a,X_t} \geq \frac{1}{2}$, we get
    \[
       \forall s'\in\R;\quad \Sigma_{X,t}(s') = \pth{1 + s'}\wedge\frac{1}{2}.
    \]
    
  \item If $a(X_t(\omega)) = 0$, but is not locally equal to zero, and $b(X_t(\omega)) = 0$, we similarly obtain the following inequality
    \[
      \forall s'\in\R;\quad \Sigma_{X,t}(s') = \pthbb{\frac{1}{2} + \frac{\Sigma_{a^2(X_\sbullet),t}(2s')}{2}}\wedge\pthB{1 + \Sigma_{b(X_\sbullet),t}(s')}\wedge\frac{1}{2},
    \]
    as $b(X_t(\omega))=0$ and using the translation property of the pseudo frontier. This inequality leads to
    \[
      \forall s'\in\R;\quad \Sigma_{X,t}(s') = \pthbb{\frac{1}{2} + \psd{\alpha}_{a,X_t} \psd{\alpha}_{X,t} + s'}\wedge\pthb{1 + \psd{\alpha}_{b,X_t} \psd{\alpha}_{X,t} + s'}\wedge\frac{1}{2}
    \]
    and in particular
    \[
      \psd{\alpha}_{X,t} \geq \pthbb{\frac{1}{2} + \psd{\alpha}_{a,X_t} \psd{\alpha}_{X,t}}\wedge\pthb{1 + \psd{\alpha}_{b,X_t} \psd{\alpha}_{X,t}}.
    \]
    Therefore, if $\psd{\alpha}_{a,X_t} < 1$ or $\psd{\alpha}_{b,X_t} < 1$, we obtain:
    \[
      \psd{\alpha}_{X,t} \geq \pthbb{\frac{1}{2(1 - \psd{\alpha}_{a,X_t})_+}}\wedge \pthbb{\frac{1}{(1 - \psd{\alpha}_{b,X_t})_+}},
    \]
    which leads to
    \begin{equation} \label{eq:sde3}
  	  \forall s'\in\R;\quad  \Sigma_{X,t}(s') \geq \pthbb{\frac{1}{2(1 - \psd{\alpha}_{a,X_t})_+} + s'}\wedge\pthbb{\frac{1}{(1 - \psd{\alpha}_{b,X_t})_+} + s'}\wedge\frac{1}{2}.
  	\end{equation}
    In the case $\psd{\alpha}_{a,X_t} \geq 1$ and $\psd{\alpha}_{b,X_t} \geq 1$, the previous inequalities imply that $\psd{\alpha}_{X,t} = +\infty$ and $\sigma_{X,t} = \Sigma_{X,t} = \frac{1}{2}$.
  \item Finally, if $a(X_\sbullet(\omega))$ is locally equal to zero at $t$, we locally obtain a classic differential equation $\dt X_t = b(X_t) \dt t$. Similarly to previous calculations, we get the following inequality:
    \begin{equation} \label{eq:sde4}
  	  \forall s'\in\R;\quad \Sigma_{X,t}(s') \geq \pthbb{\frac{1}{(1 - \alpha_{b,X_t})_+} + s'}\wedge 1.
  	\end{equation}
\end{enumindent}

We notice in the previous results that a self-regulating aspect appears in the regularity of diffusions, since the 2-microlocal frontier at $t$ can depend on the value of $a(X_t(\omega))$ and $b(X_t(\omega))$. The behaviour is also illustrated in the two following Examples \ref{ex:BESQ} and \ref{ex:heston}.

For sake of readability, we have only considered homogeneous diffusions. But, results could be easily extended to non-homogeneous diffusions
$
  \dt X_t = a(t,X_t,\omega) \,\dt B_t + b(t,X_t,\omega) \,\dt t.
$
In that case, lower bounds obtained would depend on the stochastic H\"older exponents of $a$ and $b$ along the two directions, i.e. $\pth{ \psd{\alpha}_{a(\sbullet,X_t),t}, \psd{\alpha}_{a(t,\sbullet),X_t} }(\omega)$ and  $\pth{ \psd{\alpha}_{b(\sbullet,X_t),t}, \psd{\alpha}_{b(t,\sbullet),X_t} }(\omega)$.

More generally, we can see that the previous methodology can be extended to any kind of stochastic differential equation to get, at least, a lower bound for the local regularity of its solution.\vspar

To end this section, we apply theses techniques to obtain the exact regularity in two different examples of SDE.
In the first one, we characterize the 2-microlocal frontier of the square of $\delta$-dimensional Bessel processes, which generalizes the result on $t\mapsto B_t^2$ described in the preliminaries.
\begin{example} \label{ex:BESQ}
  We know that for every $\delta > 0$ and $x\geq 0$, there exists a unique strong continuous solution to the following equation (see \cite{Revuz.Yor(1999)}):
  $
    Z_t = x + 2\int_0^t \sqrt{Z_s} \dt \beta_s + \delta t,
  $
  which is called the square of $\delta$-dimensional Bessel process started at $x$ ($BESQ^ \delta(x)$).
  
  Hence, for every $\delta> 0$ and $x\geq 0$, there exists an event $\Omega_0$ such that $\pr{\Omega_0}=1$ and for all $\omega\in\Omega_0$ and $t\in\R_+$, the 2-microlocal frontier of the $BESQ^ \delta(x)$ $Z$ is determined in the two following cases.
  \begin{enumerate}
  	\item If $Z_t(\omega) \neq 0$,
  	\[
  	  \forall s'\in\R;\quad \sigma_{Z,t}(s') = \Sigma_{Z,t}(s') = \pthbb{\frac{1}{2} + s'}\wedge\frac{1}{2},
  	\]
  	and in particular $\alpha_{Z,t} = \widetilde\alpha_{Z,t} = \tfrac{1}{2}$.
  	\item If $Z_t(\omega) = 0$, the pseudo frontier satisfies the equation
  	\[
      \Sigma_{Z,t}(s') = \ptha{\frac{1}{2} + \frac{\Sigma_{Z,t}(2s')}{2}}\wedge\pthb{1+s'}\wedge\frac{1}{2},
    \]
    which has a unique solution,
  	\[
  	  \forall s'\in\R;\quad \sigma_{Z,t}(s') = \Sigma_{Z,t}(s') = \pthb{1 + s'}\wedge\frac{1}{2},
  	\]
  	and in particular $\alpha_{Z,t} = 1$ and $\widetilde\alpha_{Z,t} = \tfrac{1}{2}$.
  \end{enumerate}\vsp
  
  These equalities are simple consequences of previous results on diffusions. Indeed, since $\alpha_{a,x} = \frac{1}{2}$ when $x=0$, Equations \eqref{eq:sde1} and \eqref{eq:sde2} prove the equalities on the pseudo frontier.
  
  To obtain the classic 2-microlocal frontier, we can apply Corollary \ref{cor:eq_cl_psd} on the stochastic integral $\int_0^t \!\! \sqrt{Z_s} \,\dt \beta_s$ and note that according to Definition \ref{eq:def_2ml_spaces}, the frontier is not modified when a polynomial (i.e. $\delta t$) is added. 
\end{example}

\begin{remark}
  As proved in \cite{Revuz.Yor(1999)}, if $\delta\geq 2$, the set $\brc{0}$ is transient and therefore, almost surely for any $t\in\R_+$, the 2-microlocal frontier is equal to 
  \[
  	\forall s'\in\R;\quad \sigma_{Z,t}(s') = \Sigma_{Z,t}(s') = \pthbb{\frac{1}{2} + s'}\wedge\frac{1}{2}.
  \]
  On the contrary, if $\delta < 2$, we know that $\brc{0}$ is recurrent, and the distinctions between $Z_t(\omega)\neq 0$ and $Z_t(\omega)=0$ is necessary.
\end{remark}

\begin{example} \label{ex:heston}
  In the second example, we obtain the 2-microlocal frontier of the Heston model. This model has been introduced in \cite{Heston(1993)} to describe the price of an asset which has a stochastic volatility and therefore, it consists of two correlated SDE. The interesting point in this model is that the regularity of the asset $S$ depends on the value of the volatility $\nu$, as proved below.

  In the Heston model, $S$ and $\nu$ are solutions of the following system of equations:
  \[
  \begin{cases}
    &\dt S_t = \mu S_t + \sqrt{\nu_t} S_t \dt W_t^S \\
    &\dt \nu_t = \kappa(\theta - \nu_t) \dt t + \xi \sqrt{\nu_t} \dt W_t^\nu,
  \end{cases}
  \]
  where $\mu, \kappa, \theta$ and $\xi$ are deterministic positive parameters and $(W_t^S,W_t^\nu)$ is a Brownian motion where $W^S$ and $W^\nu$ may be dependent.
  
  Then, there exists an event $\Omega_0$ such that $\pr{\Omega_0}=1$ and for all $\omega\in\Omega_0$ and $t\in\R_+$, 2-microlocal frontiers of $S$ and $\nu$ at $t$ are given by
  \begin{enumerate}
  	\item if $\nu_t(\omega) \neq 0$,
  	\[
  	  \forall s'\in\R;\quad \sigma_{S,t}(s') = \sigma_{\nu,t}(s') = \pthbb{\frac{1}{2} + s'}\wedge\frac{1}{2};
  	\]
  	\item if $\nu_t(\omega) = 0$,
  	\[
  	  \forall s'\in\R;\quad \sigma_{S,t}(s') = \sigma_{\nu,t}(s') = \pthb{1 + s'}\wedge\frac{1}{2}.
  	\]
  \end{enumerate}

  We obtain the 2-microlocal frontier of $\nu$ similarly to the previous example.
  Then, if we consider $S$, as $S$ is positive $\sqrt{\nu_t} S_t = 0$ if and only if $\nu_t = 0$. Therefore, in this case we obtain the equation
	\[
    \Sigma_{S,t}(s') = \ptha{\frac{1}{2} + \frac{\Sigma_{\nu,t}(2s')}{2}}\wedge\pthb{1+s'}\wedge\frac{1}{2},
  \]
  which leads to the expected solution.
\end{example}

\begin{remark}
  According to \cite{Albrecher.Schoutens.ea(2007)}, if the parameters satisfy the inequality $2\kappa\theta \geq \xi^2$, then the process $\nu$ is almost surely positive. Therefore, similarly to square of Bessel processes, the result is simplified in that case as $\nu_t=0$ almost surely never happens.
\end{remark}

These two examples show that it is possible to get an exact result, more precise than lower bounds obtained at the beginning of the section, when the expressions of functions $a$ and $b$ are explicit and when we are able to characterize the regularity of compositions $a\circ X$ and $b\circ X$.

Another point illustrated by these examples and previous results on diffusions is that the regularity of the solutions of these SDE are almost everywhere equal to $\frac{1}{2} + s'$ and change only at exceptional points (e.g. zeros of BESQ$^\delta$). Even if these equations were driven by martingales other than the Brownian motion, according to Theorems \ref{th:mg_vq} and \ref{th:2ml_stoch_int}, interesting regularities would still be exceptional. Nevertheless, many natural phenomenons have a regularity different from $\frac{1}{2}$ (see e.g. \cite{Feder(1988)}), and thus can not be represented by these classic diffusions. Therefore, an interesting next step would be to extend our work to stochastic differential equations driven by fractional (and multifractional) Brownian motion (see e.g. \cite{Nualart(2003)}), in order to exhibit the possibilities it brings in term of regularity.

\appendix
\section{Proofs of technical results} \label{sec:appendix}

Appendices gather some deterministic and technical results which are used along the article. 

\subsection{Pseudo 2-microlocal analysis} \label{sec:app_pseudo}

Propositions and theorems related to pseudo 2-microlocal spaces and frontier are proved in this appendix. The first result is an extension of a well-known property on local H\"older exponents: values of the local exponent around a point $t_0$ have an influence on the pseudo 2-microlocal frontier of $f$ at $t_0$.

\begin{lemma} \label{lemma:pseudo_2ml_4}
  Let $f$ be a continuous function and $t_0$ be in \R. Then, the pseudo 2-microlocal frontier at $t_0$ satisfies:
  \[
    \forall s'\in\R;\quad \Sigma_{f,t_0}(s') \leq \liminf_{u\rightarrow t_0,u\neq t_0} \widetilde{\psd{\alpha}}_{f,u},
  \]
  where $\widetilde{\psd{\alpha}}_{f,u}$ is the pseudo local H\"older exponent of $f$ at $u$.
\end{lemma}
\begin{proof}
  Let $t_0$ be in \R. There exists a sequence $(u_n)_{n\in\N}$ which converges to $t_0$ and such that 
  \[
    \lim_{n\in\N} \widetilde{\psd{\alpha}}_{f,u_n} = \liminf_{u\rightarrow t_0,u\neq t_0} \widetilde{\psd{\alpha}}_{f,u} \eqdef \alpha_0.
  \]
  According to the definition of the pseudo local H\"older exponent, we know that for all $\eps>0$, $\rho > 0$ and $n\in\N$,
  \[
    \sup_{u,v\in B(u_n,\rho)} \frac{\absb{f(u) - f(v)}}{\abs{u-v}^{\widetilde{\psd{\alpha}}_{f,u_n}+\eps}} = +\infty.
  \]
  We use this equality to prove the upper bound on the pseudo 2-microlocal frontier.
  
  Let first note that when $s'\geq 0$, for all $u,v\in B(u_n,\rho)$,
  \[
    \frac{1}{ \pthb{\abs{u-t_0}+\abs{v-t_0}}^{-s'} } \geq \pth{ 2\abs{u_n-t_0} - 2\rho }^{s'} > \abs{ u_n - t_0 }^{s'},
  \]
  when $\rho$ is small enough. Therefore, according to the first equation, for all $n\in\N$, $s'\geq 0$ and $\rho,\eps>0$
  \[
    \sup_{u,v\in B(u_n,\rho)} \frac{\absb{f(u) - f(v)}}{\abs{u-v}^{\widetilde{\psd{\alpha}}_{f,u_n}+\eps} \pthb{\abs{u-t_0}+\abs{v-t_0}}^{-s'}} = +\infty,
  \]
  which leads to
  \[
    \sup_{u,v\in B(t_0,2\abs{u_n-t_0})} \frac{\absb{f(u) - f(v)}}{\abs{u-v}^{\widetilde{\psd{\alpha}}_{f,u_n}+\eps} \pthb{\abs{u-t_0}+\abs{v-t_0}}^{-s'}} = +\infty,
  \]
  proving the expected inequality as $\lim_{n\rightarrow+\infty} \widetilde{\psd{\alpha}}_{f,u_n} = \psd{\alpha}_0$.
  
  The case $s'\leq 0$ is treated similarly, using the inequality $\pthb{\abs{u-t_0}+\abs{v-t_0}}^{s'} \geq 1$  when $\abs{u_n-t_0}$ and $\rho$ are sufficiently small.
\end{proof}

We know prove Theorem \ref{th:classic_pseudo_frontiers} which links up the pseudo 2-microlocal frontier to the classic one in the following way:
\[
  \forall s'\in\R ;\quad \Sigma_{f,x_0}(s') = \sigma_{f,x_0}(s')\wedge\pth{s'+p_{f,x_0}}\wedge 1,
\]
unless $f$ is locally constant at $t_0$.

\begin{proof}[Proof of Theorem \ref{th:classic_pseudo_frontiers}]
  
  Let first note that in the particular case $f$ is locally constant at $t_0$, we easily observe that both classic and pseudo 2-microlocal frontiers are equal to $+\infty$. Therefore, we suppose from now that $f$ is not locally constant.

  \begin{enumindent}
	  \item Let first prove the following inequality
	  \[
      \forall s'\in\R ;\quad \Sigma_{f,t_0}(s') \leq 1.
    \]
    
    We proceed by contradiction: if there exist $\eps>0$ and $s'\in\R$ such that $\Sigma_{f,t_0}(s') \geq 1+2\eps$, then, there exist $\rho>0$ and $C>0$ such that
    \[
      \forall u,v\in B(t_0,\rho);\quad \abs{f(u) - f(v)} \leq C \abs{u-v}^{1+\eps} \pthb{\abs{u-t_0}+\abs{v-t_0}}^{-s'}.
    \]
    In particular, for each $k\in\N^*$, there exists $C_k>0$ such that
    \[
      \forall u,v\in\ivff{t_0+\tfrac{\rho}{k},t_0+\rho} ;\quad \abs{f(u) - f(v)} \leq C_k \abs{u-v}^{1+\eps},
    \]
    proving that $f$ is differentiable with null derivatives on the interval $\ivff{t_0+\tfrac{\rho}{k},t_0+\rho}$. As it is satisfies for every $k\in\N^*$ and $f$ is continuous at $t_0$, it must be constant on the interval $\ivff{t_0-\rho,t_0+\rho}$, which is in contradiction with the assumption of $f$ not locally constant.
	
	  \item Let now show the inequality
	  \[
      \forall s'\in\R ;\quad \Sigma_{f,t_0}(s') \geq \sigma_{f,t_0}(s')\wedge\pth{s'+p_{f,t_0}}\wedge 1,
    \]
    where
    $
      p_{f,t_0} = \inf\brcb{n\geq 1 : f^{(n)}(t_0)\text{ exists and }f^{(n)}(t_0)\neq 0}.
    $
    
    Let $s'\in\R$ and let first assume that $\Sigma_{f,t_0}(s') \geq 0$ (the general case is explained at the end of the proof). For all $\sigma < \sigma_{f,t_0}(s')$, there exist $\rho>0$ and a polynomial $P$ such that for all $u,v\in B(t_0,\rho)$,
    \[
      \absb{ (f(u)-P(u)) - (f(v)-P(v)) } \leq \abs{u-v}^\sigma \pthb{\abs{u-t_0}+\abs{v-t_0}}^{-s'}.
    \]
    It has been observed in \cite{Echelard(2006)} that  $P$ can be chosen as the Taylor expansion of $f$ at $t_0$,
    \[
      P(u) = \sum_{k=0}^{N} a_k (u - t_0)^k.
    \]
    Furthermore, it is also proved in \cite{Echelard(2006)} that for each $k\in\N^*$ and all $s'\in\R$,
    \[
      \absb{ (u-t_0)^k - (v-t_0)^k } \leq C_k \abs{u-v}^{(k+s')\wedge 1} \pthb{\abs{u-t_0}+\abs{v-t_0}}^{-s'},
    \]
    in the neighbourhood of $t_0$.
    
    Then, according to the definition of $p_{f,t_0}$, we observe that necessarily $a_k=0$ for each $k\in\brc{1,\dotsc,p_{f,t_0}-1}$. Therefore, there exist $\rho>0$ and $C > 0$ such that for all $u,v\in B(t_0,\rho)$,
    \[
      \abs{ P(u) - P(v) } \leq C \abs{u-v}^{(p_{f,t_0}+s')\wedge 1} \pthb{\abs{u-t_0}+\abs{v-t_0}}^{-s'}.
    \]
    
    Finally, we obtain that for all $u,v\in B(t_0,\rho)$,
    \begin{align*}
      \abs{ f(u) - f(v) } 
      &\leq \absb{ (f(u)-P(u)) - (f(v)-P(v)) } + \abs{ P(u) - P(v) } \\
      &\leq C \abs{u-v}^{\sigma\wedge(p_{f,t_0}+s')\wedge 1} \pthb{\abs{u-t_0}+\abs{v-t_0}}^{-s'},
    \end{align*}
    which proves the expected inequality.
    
    \item Let finally show the converse inequality
    \[
      \forall s'\in\R ;\quad \Sigma_{f,t_0}(s') \leq \sigma_{f,t_0}(s')\wedge\pth{s'+p_{f,t_0}}\wedge 1.
    \]
    The boundedness $\Sigma_{f,t_0}(s') \leq 1$ has already been proved. Let $s'\in\R$ with $\Sigma_{f,t_0}(s') \geq 0$. 
    
    Let first suppose that $\sigma_{f,t_0}(s') > (s'+p_{f,t_0})$. Then, the function $f$ and the polynomial $P$ are such that
    \[
      \limsup_{\rho\rightarrow 0} \sup_{u,v\in B(t_0,\rho)} \frac{ (f(u)-P(u)) - (f(v)-P(v)) }{ \abs{u-v}^{s'+p_{f,t_0}} \pthb{\abs{u-t_0}+\abs{v-t_0}}^{-s'} } = 0.
    \]
    
    Furthermore, we know from \cite{Echelard(2006)} that
    \[
      \limsup_{\rho\rightarrow 0} \sup_{u,v\in B(t_0,\rho)} \frac{ \abs{ (u-t_0)^{p_{f,t_0}} - (v-t_0)^{p_{f,t_0}} } }{ \abs{u-v}^{s'+p_{f,t_0}} \pthb{\abs{u-t_0}+\abs{v-t_0}}^{-s'} } > 0.
    \]
    Since $a_{p_{f,t_0}}\neq 0$ and $a_k=0$ for each $k < p_{f,t_0}$,  we have
    \[
      \limsup_{\rho\rightarrow 0} \sup_{u,v\in B(t_0,\rho)} \frac{ \abs{ P(u)-P(v) } }{ \abs{u-v}^{s'+p_{f,t_0}} \pthb{\abs{u-t_0}+\abs{v-t_0}}^{-s'} } > 0,
    \]
    which implies
    \[
      \limsup_{\rho\rightarrow 0} \sup_{u,v\in B(t_0,\rho)} \frac{ \abs{f(u)-f(v)} }{ \abs{u-v}^{s'+p_{f,t_0}} \pthb{\abs{u-t_0}+\abs{v-t_0}}^{-s'} } > 0,
    \]
    and thus $\Sigma_{f,t_0}(s') \leq \pth{s'+p_{f,t_0}} = \sigma_{f,t_0}(s')\wedge\pth{s'+p_{f,t_0}}$. The case $\sigma_{f,t_0}(s') < (s'+p_{f,t_0})$ is treated similarly. 
    
    Finally, let suppose $\sigma_{f,t_0}(s') = s'+p_{f,t_0}$. 
    As $P$ corresponds to the Taylor expansion of $f$ at $t_0$ and for each $k < p_{f,t_0}$, $a_k=0$, we have
    \[
      \lim_{u\rightarrow t_0} \frac{ f(u) - f(t_0) } {(u-t_0)^{p_{f,t_0}}} = \lim_{u\rightarrow t_0} \frac{ f(u) - f(t_0) } {(u-t_0)^{p_{f,t_0}+s'}(u-t_0)^{-s'}} = a_{p_{f,t_0}} \neq 0,
    \]
    proving that $\Sigma_{f,t_0}(s') \leq \pth{s'+p_{f,t_0}}$.
    
    \item We have proved that $\Sigma_{f,t_0}(s') = \sigma_{f,t_0}(s')\wedge\pth{s'+p_{f,t_0}}\wedge 1$ for all $s'\in\R$ such that $\Sigma_{f,t_0}(s')\geq 0$.
    
    Let $k\in\N^*$ and $s'\in\R$ such that $\Sigma_{f,t_0}(s')\in\ivff{-k,-k+1}$. According to Definitions~\ref{eq:def_2ml_spaces2} and \ref{def:pseudo_spaces}, we observe that both pseudo and classic 2-microlocal frontiers are defined using increments of the $k^{th}$ integration of $f$ at $t_0$: $I_{t_0}^k(f)$. Based on this remark, we can adapt the previous reasoning to every $I_{t_0}^k(f)$, $k\in\N$, therefore proving the equality for all $s'\in\R$.
    \end{enumindent} 
\end{proof}

We end this section with the proof of Theorem \ref{th:int_pseudo_frontier} related to the behaviour of the pseudo 2-microlocal frontier when a function is integrated.
\begin{proof}[Proof of Theorem \ref{th:int_pseudo_frontier}]
  Recall that $F$ is defined by 
  \[
    \forall t\in\R;\quad F(t) = \int_0^t f(s) \,\dt s.
  \]
  If $f$ is locally equal to $0$ at $t_0$, $F$ is clearly constant in some neighbourhood of $t_0$, and therefore, both frontiers of $f$ and $F$ are equal to $+\infty$.
  
  \begin{enumindent}
	  \item Then, let consider the case $f(t_0)\neq 0$. Without any loss of generality, we can suppose $f(t_0) > 0$. Since $f$ is continuous, there exist $\rho,C_1,C_2>0$ such that for all $u,v\in B(t_0,\rho)$,
	  \[
      C_1 \abs{ u - v} \leq \abs{ F(v)-F(u) } \leq C_2 \abs{u-v}.
    \]
    $F$ locally behaves like the function $x\mapsto x$, and thus has the following pseudo frontier 
    \[
      \forall s'\in\R;\quad \Sigma_{F,t_0}(s') = (1+s')\wedge 1.
    \]
    \item Finally, let suppose $f(t_0) =  0$, but is not locally equal to $0$.
    
    We know that the classic 2-microlocal frontier of $F$ satisfies $\sigma_{F,t_0} = \sigma_{f,t_0}+1$. Furthermore, as $f$ is continuous, $F'=f$ and therefore, as $F'(t_0)=0$, the exponent $p_{F,t_0}$ defined previously is simply equal to $p_{f,t_0}+1$.
  
  Theorem \ref{th:classic_pseudo_frontiers} implies for all $s'\in\R$,
  \[
    \Sigma_{f,t_0}(s') = \sigma_{f,t_0}(s')\wedge\pth{s'+p_{f,t_0}+1}\wedge 1.
  \]
  Therefore, we obtain
  \begin{align*}
    \Sigma_{F,t_0}(s') &= \sigma_{F,t_0}(s')\wedge\pth{s'+p_F,t_0+1}\wedge 1 \\
    &= \pthb{\sigma_{f,t_0}(s')+1}\wedge\pth{s'+p_{f,t_0}+2}\wedge 2\wedge 1 \\
    &= \pthB{\pthb{\sigma_{f,t_0}(s')\wedge\pth{s'+p_{f,t_0}+1}\wedge 1}+1}\wedge 1 \\
    &= \pthb{ \Sigma_{f,t_0}(s') +1 }\wedge 1.
  \end{align*}
  \end{enumindent}
\end{proof}

\subsection{Martingale example} \label{sec:app_mg}

Example \ref{ex:mg1} describes a martingale with particular local regularity. This section details the construction of the quadratic variation which leads to this interesting 2-microlocal frontier.

\begin{lemma} \label{lemma:function}
  There exists a deterministic non-decreasing function $f_\alpha$ which has the following pseudo 2-microlocal frontier at $0$
  \[
    \forall s'\geq -1;\quad \Sigma_{f_\alpha,0}(s') = \ptha{ \frac{s'+ 1}{1-\log_2(\alpha) } }\wedge 1,
  \]
  where $\alpha\in\ivff{0,1}$ and $\log_2(\alpha) = \frac{\log(\alpha)}{\log(2)}$.
\end{lemma}

\begin{proof}
  The construction of this example is similar to Cantor's function. Beginning with the identity $f_0:x\mapsto x$, we iterate the construction using the following process. Let first choose a parameter $\alpha\in\ivff{0,1}$. At each step $N$, the function $f_{N-1}$ is equal to the identity on the interval $\ivff{0,\frac{1}{2^N}}$. Then, $f_N$ is obtained by modifying $f_{N-1}$ on this interval as described in figure \ref{fig:ex_function}.
  \begin{figure}[!ht]
    \centering
    \includegraphics[scale=0.4]{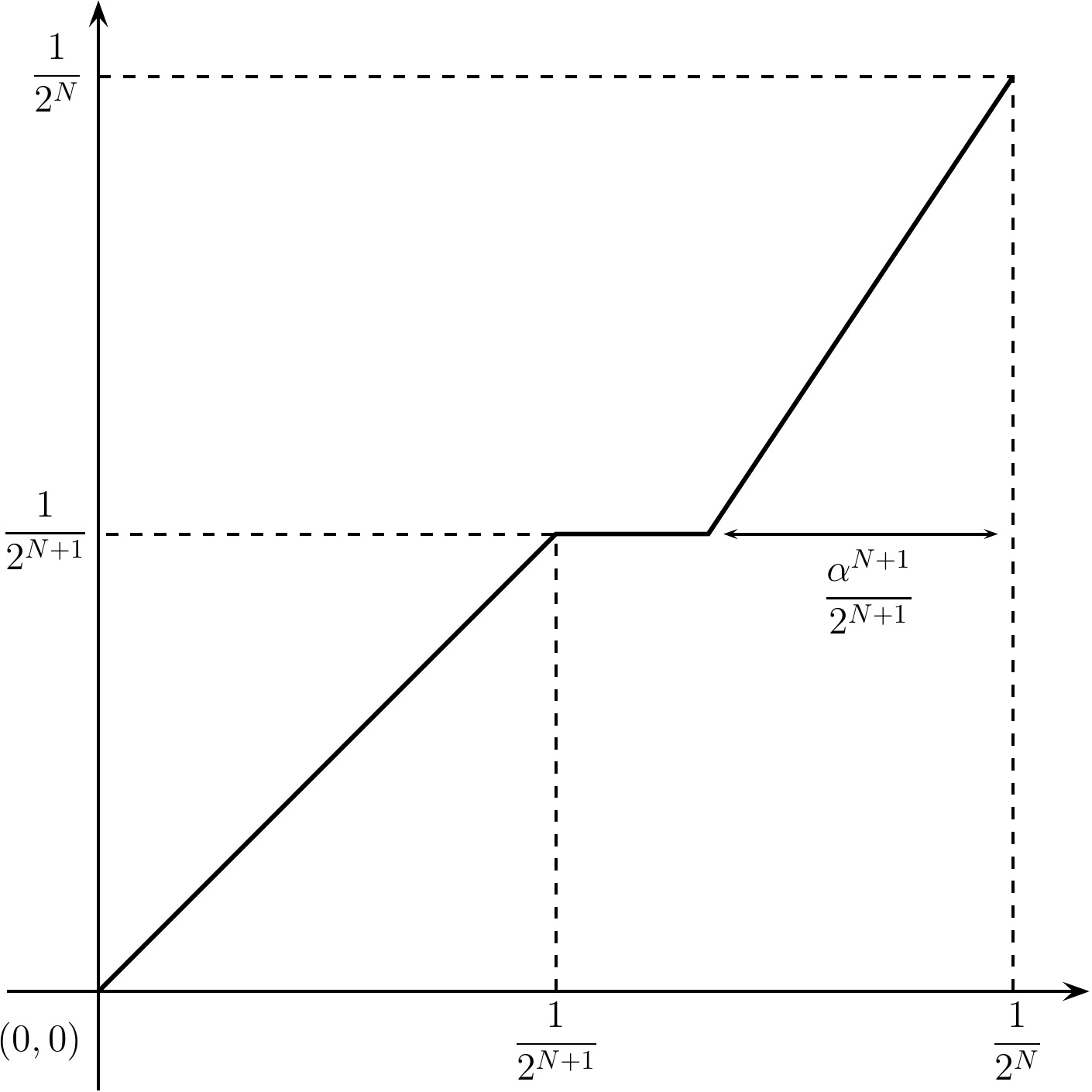}
    \caption{Function $f_N$ obtained on the interval $\ivff{0,\frac{1}{2^N}}$.}
    \label{fig:ex_function}
  \end{figure}

  The sequence $(f_N)_{N\in\N}$ uniformly converges to a a continuous function $f$ since it satisfies a Cauchy criterion
  \[
    \forall n,p\in\N;\quad \norm{f_{n+p} - f_n}_\infty \leq 2^{-n}.
  \]

  Let now describe the pseudo 2-microlocal frontier of this function $f$ at $0$.
  \begin{enumindent}
    \item We first consider one-side inequality,
    \[
      \forall s'\geq -1;\quad \Sigma_{f_\alpha,0}(s') \leq \ptha{ \frac{s'+ 1}{1-\log_2(\alpha) } }\wedge 1.
    \]
    As usually when one wants to obtain an upper bound for a regularity exponent, specific sequences $(s_n)_{n\in\N}$ and $(t_n)_{n\in\N}$ are constructed to capture the irregularity of the function. More precisely, let $t_n=2^{-n}$ and $s_n=t_n - \pthb{\frac{\alpha}{2}}^{n+1}$. Then, for every $n\in\N$, $\sigma\in\ivff{0,1}$ and $s'\geq -1$, we have
    \begin{align*}
      \abs{ f(s_n)-f(t_n) } 
      &= 2^{-(n+1)} 
      = 2^{-(n+1)(1+s')} \cdot 2^{s'(n+1)} \\
      &= \abs{t_n - s_n }^{\frac{1+s'}{1-\log_2(\alpha)}} \cdot 2^{s'(n+1)} 
      \geq C \abs{t_n - s_n }^{\frac{1+s'}{1-\log_2(\alpha)}} \pthb{\abs{t_n} + \abs{s_n}}^{-s'},
    \end{align*}
    which shows this first inequality.
    
    \item Let now prove the converse inequality. Let $s'\geq -1$ and $\sigma < \ptha{ \frac{s'+ 1}{1-\log_2(\alpha) } }\wedge 1$. Then, we have to show that there exist $\rho > 0$ and $C>0$ such that
    \[
      \forall u,v \in B(0,\rho);\quad \abs{ f(u) - f(v) } \leq C\abs{u-v}^\sigma \pthb{ \abs{u}+\abs{v} }^{-s'}.
    \]
    
    Let first observe that for all $u\in\R_+$, we have $f(u) \leq u$. 
    
    Let $u\leq v\in\ivff{0,1}$ and $m\in\N$ such that $2^{-(m+1)} \leq v \leq 2^{-m}$. We distinguish two different cases.
    \begin{itemize}
	    \item If $u \leq 2^{-(m+2)}$, then $2^{-(m+2)} \leq \abs{u-v} \leq 2^{-(m-1)}$. Therefore for all $s'\geq -1$, we obtain
	    \begin{align*}
	      \abs{ f(u) - f(v) } 
	      &\leq f(u) + f(v) \\
	      &\leq 2^{-(m-1)} 
	      = 2^{-(m-1)(1+s')} \cdot 2^{s'(m-1)} \\
	      &\leq C \abs{u-v}^{1+s'} \pthb{ \abs{u}+\abs{v} }^{-s'} 
	      \leq C \abs{u-v}^\sigma \pthb{ \abs{u}+\abs{v} }^{-s'},
	    \end{align*}
	    since $\sigma < \ptha{ \frac{s'+ 1}{1-\log_2(\alpha) } }\wedge 1 \leq 1+s'$ and $\abs{u-v}\in\ivff{0,1}$.
	    
	    \item If $u \geq 2^{-(m+2)}$. For sake of simplicity, we suppose $u \geq 2^{-(m+1)}$ (the general case is as simple, but longer to argue). Then, according to the construction of the function $f_n$, we observe that 
	    \begin{align*}
	      \abs{f(u)-f(v)} 
	      &\leq \abs{f(u)-f(v)}^{\sigma} \cdot \abs{f(u)-f(v)}^{1-\sigma} \\
	      &\leq \pthb{ \abs{u-v}\alpha^{-(n+1)} }^\sigma \cdot \pthb{f(u)+f(v)}^{1-\sigma} \\
	      &\leq \abs{u-v}^\sigma 2^{-(n+1)\sigma\log_2(\alpha)} \pthb{ \abs{u}+\abs{v} }^{1-\sigma} \\
	      &\leq C \abs{u-v}^\sigma \pthb{ \abs{u}+\abs{v} }^{1-\sigma+\sigma\log_2(\alpha)} \\
	      &\leq C \abs{u-v}^\sigma \pthb{ \abs{u}+\abs{v} }^{-s'},
	    \end{align*}
	    since $-s' \leq 1-\sigma+\sigma\log_2(\alpha)$.
    \end{itemize}
    Based on these different cases, we obtain the inequality
    \[
      \forall s'\geq -1;\quad \Sigma_{f_\alpha,0}(s') \geq \ptha{ \frac{s'+ 1}{1-\log_2(\alpha) } }\wedge 1,
    \]
    which concludes the proof.
  \end{enumindent}
  
   We note that the pointwise H\"older $\alpha_{f,0}$ exponent does not depend on $\alpha$ and is always equal to $1$, whereas the local H\"older exponent $\widetilde\alpha_{f,0}$ is equal to $\frac{1}{1-\log_2(\alpha)}$. Therefore in the case of $\alpha=1$, the function is simply the identity and $\widetilde\alpha_{f,0}=1$ and on the contrary, when $\alpha=0$, $\widetilde\alpha_{f,0}=0$ and the function $f$ has jumps at every $\frac{1}{2^n}$ (but is still continuous at $0$).
\end{proof}


\end{document}